\documentclass{amsproc}

\usepackage{amsmath,amssymb,xcolor,graphicx}
\usepackage[all]{xy}
\usepackage{colonequals}
\usepackage{mathrsfs}
\usepackage{nicefrac}
\usepackage[inline, shortlabels]{enumitem}
\usepackage{comment}
\usepackage{float}
\usepackage{tikz}
\usepackage{tikz-cd}
\usepackage{multirow}

\usepackage[breaklinks,backref=page]{hyperref}
\hypersetup{colorlinks=true,urlcolor=blue,citecolor=blue,linkcolor=blue}
\usepackage{cleveref}

\Crefname{equation}{}{}

\newcommand{\psmod}[1]{~(\textup{\text{mod}}~{#1})}

\numberwithin{equation}{subsection}

\theoremstyle{plain}
\newtheorem{lemma}[equation]{Lemma}
\newtheorem{prop}[equation]{Proposition}

\newtheorem{corollary}[equation]{Corollary}
\newtheorem{cor}[equation]{Corollary}
\newtheorem{theorem}[equation]{Theorem}
\newtheorem{thm}[equation]{Theorem}

\theoremstyle{remark}
\newtheorem{example}[equation]{Example}
\newtheorem{remark}[equation]{Remark}

\theoremstyle{definition}
\newtheorem{definition}[equation]{Definition}

\newenvironment{enumalph}
{\begin{enumerate}\renewcommand{\labelenumi}{\textnormal{(\alph{enumi})}}}
{\end{enumerate}}

\newenvironment{enumroman}
{\begin{enumerate}\renewcommand{\labelenumi}{\textnormal{(\roman{enumi})}}}
{\end{enumerate}}

\newcommand{\C}{\mathbb C}

\newcommand{\PP}{\mathbb P}
\newcommand{\Q}{\mathbb Q}
\newcommand{\R}{\mathbb R}

\newcommand{\Z}{\mathbb Z}

\newcommand{\bbP}{\mathbb{P}}
\newcommand{\bbQ}{\mathbb{Q}}
\newcommand{\bbR}{\mathbb{R}}

\newcommand{\bbZ}{\mathbb{Z}}

\newcommand{\frakb}{\mathfrak b}

\newcommand{\frakd}{\mathfrak d}

\newcommand{\frakf}{\mathfrak f}
\newcommand{\frakg}{\mathfrak g}

\newcommand{\frakM}{\mathfrak M}
\newcommand{\frakN}{\mathfrak N}
\newcommand{\frakp}{\mathfrak p}

\newcommand{\frakQ}{\mathfrak Q}
\newcommand{\fraks}{\mathfrak s}

\newcommand{\p}{\mathfrak p}

\newcommand{\calH}{\mathcal H}
\newcommand{\calO}{\mathcal O}
\newcommand{\calP}{\mathcal P}
\newcommand{\calR}{\mathcal R}
\newcommand{\calU}{\mathcal U}
\newcommand{\calV}{\mathcal V}

\newcommand{\eps}{\varepsilon}

\newcommand{\legen}[2]{\left(\frac{#1}{#2}\right)}

\newcommand{\inv}{^{-1}}

\newcommand{\Shat}{\widehat{S}}
\newcommand{\Fhat}{\widehat{F}}
\newcommand{\calOhat}{\widehat{\mathcal O}}

\newcommand{\Zhat}{\widehat{\mathbb Z}}

\DeclareMathOperator{\Cl}{Cl}

\let\det\relax % "Undefine" \det to have \det^{-1} behave properly

\DeclareMathOperator{\det}{det}
\DeclareMathOperator{\disc}{disc}
\DeclareMathOperator{\Emb}{Emb}

\DeclareMathOperator{\Frob}{Frob}
\DeclareMathOperator{\Gal}{Gal}
\DeclareMathOperator{\Gen}{Gen}
\DeclareMathOperator{\GL}{GL}
\DeclareMathOperator{\Hilb}{Hilb}
\DeclareMathOperator{\impart}{Im}

\DeclareMathOperator{\M}{M}
\DeclareMathOperator{\nrd}{nrd}

\DeclareMathOperator{\ord}{ord}
\DeclareMathOperator{\rk}{rk}
\DeclareMathOperator{\rmP}{P}
\DeclareMathOperator{\PGL}{PGL}
\DeclareMathOperator{\Pl}{Pl}

\DeclareMathOperator{\sgn}{sgn}
\DeclareMathOperator{\SL}{SL}

\DeclareMathOperator{\tr}{tr}
\DeclareMathOperator{\Typ}{Typ}

\DeclareMathOperator{\Nm}{Nm}
\DeclareMathOperator{\Pic}{Pic}

\DeclareMathOperator{\Stab}{Stab}
\DeclareMathOperator{\vol}{vol}

\newcommand{\iso}{\simeq}

\newcommand{\defi}{\textsf}

\definecolor{jdr}{RGB}{195,103,224}

\definecolor{ForestGreen}{rgb}{0.13,0.54,0.13}

\begin{document}    

\title[Hilbert modular surfaces]{A database of basic numerical invariants \\ of Hilbert modular surfaces}

\author{Eran Assaf}
\address{Department of Mathematics, Dartmouth College, 6188 Kemeny Hall, Hanover, NH 03755, USA}
\curraddr{}
\email{assaferan@gmail.com}
\urladdr{\url{http://www.math.dartmouth.edu/~eassaf/}}

\author{Angelica Babei}
\address{Department of Mathematics $\&$ Statistics, McMaster University, Hamilton Hall, 1280 Main Street West, Hamilton, ON,
L8S 4K1, Canada}
\curraddr{}
\email{babeia@mcmaster.ca}
\urladdr{\url{https://angelicababei.com}}

\author{Ben Breen}
\address{School of Mathematics and Statistics, Clemson University, Clemson, SC, 29631}
\curraddr{}
\email{benjaminkbreen@gmail.com}
\urladdr{\url{www.benbreenmath.com}}

\author{Edgar Costa}
\address{Department of Mathematics, Massachusetts Institute of Technology, Cambridge, MA 02139-4307, USA}
\curraddr{}
\email{edgarc@mit.edu}
\urladdr{\url{https://edgarcosta.org}}

\author{Juanita Duque-Rosero}
\address{Department of Mathematics, Dartmouth College, 6188 Kemeny Hall, Hanover, NH 03755, USA}
\curraddr{}
\email{juanita.gr@dartmouth.edu}
\urladdr{\url{https://math.dartmouth.edu/~jduque/}}

\author{Aleksander Horawa}
\address{Mathematical Institute, University of Oxford, Woodstock Road,
Oxford,
OX2 6GG, UK}
\curraddr{}
\email{horawa@maths.ox.ac.uk}
\urladdr{\url{https://people.maths.ox.ac.uk/horawa/}}

\author{Jean Kieffer}
\address{Department of Mathematics, Harvard University, 1 Oxford St., Cambridge, MA 02138, USA}
\curraddr{}
\email{kieffer@math.harvard.edu}
\urladdr{\url{https://scholar.harvard.edu/kieffer}}

\author{Avinash Kulkarni}
\address{Department of Mathematics, Dartmouth College, 6188 Kemeny Hall, Hanover, NH 03755, USA}
\curraddr{}
\email{avinash.a.kulkarni@dartmouth.edu}
\urladdr{\url{https://math.dartmouth.edu/~akulkarn/}}

\author{Grant Molnar}
\address{Department of Mathematics, Dartmouth College, 6188 Kemeny Hall, Hanover, NH 03755, USA}
\curraddr{}
\email{Grant.S.Molnar.GR@dartmouth.edu}
\urladdr{\url{https://www.grantmolnar.com}}

\author{Sam Schiavone}
\address{Department of Mathematics, Massachusetts Institute of Technology, Cambridge, MA 02139-4307, USA}
\curraddr{}
\email{sschiavo@mit.edu}
\urladdr{\url{https://math.mit.edu/~sschiavo/}}

\author{John Voight}
\address{Department of Mathematics, Dartmouth College, 6188 Kemeny Hall, Hanover, NH 03755, USA}
% \curraddr{}
\email{jvoight@gmail.com}
\urladdr{\url{http://www.math.dartmouth.edu/~jvoight/}}

\thanks{This research was supported by Simons Collaboration Grant (550029, to Voight).  Costa and Schiavone were supported by a Simons Collaboration Grant (550033, to Poonen and Sutherland) and Kieffer (550031, to Elkies). Breen received additional support from NSF RTG Grant DMS $\#1547399$. Horawa was supported by the NSF grant DMS-2001293 and UK Research and Innovation [grant number MR/V021931/1].}

\renewcommand{\shortauthors}{Assaf, et al.}

\begin{abstract}
We describe algorithms for computing geometric invariants for Hilbert modular surfaces, and we report on their implementation.
\end{abstract}

\maketitle

\setcounter{tocdepth}{1}
%\tableofcontents

\section{Introduction}

\subsection{Motivation}

Modular curves serve as essential motivation for the Langlands program and provide a continued and rich domain for mathematical study and explicit computation.  As we consider possible generalizations, we move up in dimension and encounter Hilbert modular surfaces.  A first step in surveying the fascinating landscape of Hilbert modular surfaces would be to organize and tabulate their basic arithmetic and geometric invariants in a sufficiently general manner.

In this paper, we begin to undertake this task.  We design and implement algorithms to compute invariants---including cusp and elliptic cycle data, Chern and Betti numbers, arithmetic genus, holomorphic Euler characteristic, and Kodaira type---and then compile data for a range of Hilbert modular surfaces.  Our effort builds upon foundational work of Hirzebruch \cite{Hirzebruch,Hir77}, Hirzebruch--van der Geer \cite{vdg-Hirzebruch}, and van der Geer \cite{vdG}, who systematically computed invariants up to discriminant $500$ restricted to level $1$ and the group $\SL_2$.  (There has also been substantial work in higher dimension, but continuing in level $1$ and for the group $\SL_2$: see the survey by Grundman \cite{Grundman}.)  Here, we generalize to nontrivial level and work with both $\SL_2$ and $\GL_2^+$, which requires new analysis.  

Our code is implemented in \textsf{Magma} \cite{Magma} and is available online \cite{ourcode}; the associated dataset is also available online \cite{ourdata}.  For the subset where our computations overlap with existing work above, we have verified that it matches.  We are in the process of including this data in the $L$-functions and Modular Forms Database (LMFDB) \cite{LMFDB}, to make it easy to browse and search.

\subsection{Organization}

The paper is organized as follows.  After setting notation in \cref{sec:prelim}, we begin with the enumeration of cusps and their resolution using Hirzebruch--Jung continued fractions in \cref{sec:cusps}, generalizing work of Dasgupta--Kakde \cite{DasguptaKakde}.  Next, in \cref{sec:elliptic_points} we turn to the enumeration of elliptic points and describe their rotation factors using the theory of embedding numbers and work of Prestel \cite{Prestel}.  In \cref{sec:dims}, we compute the 
generating series for the dimension of spaces of cusp forms.  With these three main ingredients in hand, in \cref{sec:geom_invs} we compute the desired numerical invariants, including Chern numbers, Betti numbers, and in some cases the Kodaira dimension.  We conclude in \cref{sec:data} with a description of the data tabulated, as well as directions for future work.

\subsection{Acknowledgements}
We thank Lassina Demb\'el\'e, Helen Grundman and the anonymous referees for helpful comments, as well as Sara Chari, Michael Musty, Shikhin Sethi, and Samuel Tripp for their contributions on related parallel work. 

\section{Preliminaries} \label{sec:prelim}

References for Hilbert modular forms include Freitag \cite{Freitag}, van der Geer \cite{vdG}, and Goren \cite{Goren}; for a computational introduction, see Demb\'el\'e--Voight \cite{DembVoight}.

\subsection{Basic notation}

Throughout, let $F$ be a totally real field of degree $n \colonequals [F:\Q]$ with discriminant $d_F$, and let $\Pl F$ be the set of places of $F$.  (We will eventually restrict $F$ to be a real quadratic field, but some of our results hold in this general case.)  Let $R \colonequals \Z_F$ be the ring of integers of $F$, and let $\Cl R$ be the class group of $R$ with $h=h(R) \colonequals \#\Cl R$ its cardinality.

For a real place $v \colon F \hookrightarrow \R$ and $a \in F$, we write $a_v \colonequals v(a) \in \R$. Similarly, for $\alpha \in \M_2(F)$ write $\alpha_v \in \M_2(\R)$ for the matrix obtained by applying $v$ to $a$ entrywise.  In our algorithms, we fix an ordering of real places by the roots of a \texttt{polredabs} defining polynomial for $F$. Let~$\sgn: F^\times\to \{\pm1\}^n$ be the sign map. We say $a \in F^\times$ is \defi{totally positive} if $a \in \ker \sgn$ (i.e., $a_v>0$ for all $v$), and write $F^\times_{>0}$ (resp.~$R_{>0}^\times$) for the group of totally positive elements of $F$ (resp.~totally positive units of~$R$).  Let $\Cl^+ R$ be the narrow class group of $R$ with cardinality $h^+ =h^+(R)\colonequals \#\Cl^+ R$.  There is an exact sequence
\begin{equation} \label{eqn:pm1rtimes}
1 \to \{\pm 1\}^n/\sgn(R^\times) \to \Cl^+ R \to \Cl R \to 1 
\end{equation}
so the natural surjection $\Cl^+ R \to \Cl R$ is an isomorphism if and only if there are units of $R$ of all possible signs.  

Let $\mathbf{H} \colonequals \{z \in \C : \impart z>0\}$ be the upper half-plane equipped with its hyperbolic metric and let $\calH \colonequals \mathbf{H}^n$, with the product indexed by the real places $v$.  The  group 
\begin{equation} 
\GL_2^+(F)\colonequals\{\gamma \in \GL_2(F) : \det \gamma \in F^\times_{>0}\}
\end{equation}
acts naturally by orientation-preserving isometries on $\calH$ via coordinatewise linear fractional transformations
\begin{equation} \label{eqn:zgam}
z \colonequals (z_v)_v \mapsto \gamma z \colonequals (\gamma_v z_v)_v \colonequals \left(\frac{a_v z_v+b_v}{c_v z_v + d_v}\right)_{v}. 
\end{equation}

For a nonzero ideal $\frakN \subseteq R$ and a (nonzero) fractional $R$-ideal $\frakb \subset F$, we define the \defi{standard congruence subgroups} of level $\frakN$ by
\begin{equation}  \label{eqn:gammagroups}
\begin{aligned}
\Gamma_{0}(\frakN)_\frakb &\colonequals \begin{pmatrix} R & \frakb^{-1} \\ \frakN\frakb & R \end{pmatrix} \cap \det^{-1}(R_{>0}^{\times})\\
&\ = \left\{ \begin{pmatrix} a & b \\ c & d \end{pmatrix} \in \GL_2^+(F) : 
a,d \in R,\ b \in \frakb^{-1},\ c \in \frakN\frakb,\ ad-bc \in R_{>0}^{\times} \right\}, \\
\Gamma_1(\frakN)_\frakb &\colonequals \begin{pmatrix} 1+\frakN & \frakb^{-1} \\ \frakN\frakb & R \end{pmatrix} \cap \det^{-1}(R_{>0}^{\times})
\end{aligned}
\end{equation}
so that $\Gamma_1(\frakN)_\frakb \leq \Gamma_0(\frakN)_\frakb \leq \GL_2^+(F)$.  These subgroups arise naturally, since
\begin{equation}  
\Gamma_0(1)_\frakb = \GL^+(R \oplus \frakb) 
\end{equation}
is the group of oriented $R$-module automorphisms of $R \oplus \frakb$ and $\Gamma_0(\frakN)_\frakb$  is the subgroup that stabilizes the first factor modulo $\frakN$. (Every projective $R$-module of rank $2$ is isomorphic to~$R\oplus\frakb$ for some~$\frakb$.)

For $\alpha \in F_{>0}^\times$, conjugation by $(\begin{smallmatrix} \alpha & 0 \\ 0 & 1 \end{smallmatrix})$ defines an isomorphism
\begin{equation} \label{eqn:g01bconj}
\Gamma_i(\frakN)_\frakb \iso \Gamma_i(\frakN)_{\alpha\frakb}
\end{equation}
for $i=0,1$; so up to isomorphism we may take the fractional ideals $\frakb$ among a choice of representatives of $\Cl^+ R$. (This corresponds to an oriented isomorphism $R\oplus\frakb \iso R\oplus \alpha\frakb$.) It is sometimes convenient to choose the indexing ideal $\frakb$ to be integral and a multiple of the different of $F$.

We may also take the kernel of the determinant map on these groups, giving
\begin{equation}
\Gamma_i^1(\frakN)_\frakb \colonequals \ker(\det|_{\Gamma_i(\frakN)_\frakb}) = \Gamma_i(\frakN)_\frakb \cap \SL_2(F)
\end{equation}
for $i=0,1$ and an exact sequence of groups
\begin{equation}
1 \to \Gamma_i^1(\frakN)_\frakb \to \Gamma_i(\frakN)_\frakb \to R_{>0}^\times \to 1.
\end{equation}
In view of \cref{eqn:pm1rtimes}, modulo scalars these groups agree when $\Cl^+ R = \Cl R$. 

For any subgroup $\Gamma \leq \GL_2^+(F)$, we denote by $\rmP\!\Gamma$ its image under the projection map $\GL_2^+(F) \to \PGL_2^+(F)$.  Then $\rmP\!\Gamma_i^1(\frakN)_{\frakb} \leq \rmP\!\Gamma_i(\frakN)_{\frakb}$ is a subgroup of index dividing $2^{n-1}$ (for $i=0,1$).  

\begin{remark}
We might also consider other subgroups refining the determinant, for example, matrices whose determinant is $1$ modulo $\frakN$.  These alternate families are organized by characters; we hope to pursue this in future work.  
\end{remark}

A more general definition of congruence subgroups is the following. We define the \defi{full level} subgroup (or \defi{principal congruence} subgroup) of level $\frakN$ by
\begin{equation}
\Gamma(\frakN)_\frakb \colonequals \begin{pmatrix} 1+\frakN & \frakN\frakb^{-1} \\ \frakN\frakb & 1+\frakN \end{pmatrix} \cap \det^{-1}(R^{\times}_{>0})
\end{equation}
and say $\Gamma \leq \GL_2^+(F)$ is a \defi{congruence subgroup} if~$\rmP\!\Gamma$ is conjugate to a group that contains~$\rmP\!\Gamma(\frakN)_\frakb$ for some $\frakN,\frakb$. For the most part, we restrict attention to standard congruence subgroups in this article.

\subsection{Hilbert modular varieties}
\label{sec:hilb_varieties}

Let $\Gamma < \GL_2^+(F)$ be a congruence subgroup.  Via the action \eqref{eqn:zgam}, the group $\Gamma$ is a discrete group acting properly on $\calH$.
The group $\rmP\!\Gamma$ acts faithfully on $\calH$.
The quotient
\begin{equation}
Y(\Gamma)(\C) \colonequals \Gamma \backslash \calH
\end{equation} 
is a complex orbifold of dimension $n$.  The complex analytic space $Y(\Gamma)(\C)$ is the set of complex points of a quasi-projective variety $Y(\Gamma)$, called a \defi{Hilbert modular variety}; this variety has a canonical model, defined over its reflex field (a number field).  Hilbert modular varieties admit an interpretation as a moduli space for polarized abelian surfaces with real multiplication and level structure.

The Baily--Borel compactification of $Y(\Gamma)$ is obtained by adding finitely many points as follows.
Let $\calH^* \colonequals \calH \cup \PP^1(F)$ be the completed product of upper half-planes. The action~\eqref{eqn:zgam} of $\Gamma$ on $\calH$ extends to an action on $\calH^*$.  
The \defi{cusps} of $\Gamma$ are the orbits of $\PP^1(F)$ under~$\Gamma$.  Then the analytic space
\begin{equation}
    \overline{Y}(\Gamma)(\C) \colonequals \Gamma \backslash \calH^* 
\end{equation}
obtained from adding cusps is compact, and $\overline{Y}(\Gamma)$ is a proper variety.  However, $\overline{Y}(\Gamma)$ is singular at the cusps (and any elliptic points) when $n>1$.  For $n=2$, there exists a minimal desingularization 
    \begin{equation} 
        \pi \colon X(\Gamma) \to \overline{Y}(\Gamma) 
    \end{equation}
which can be understood explicitly: algorithms for enumerating and resolving cusps and quotient singularities are investigated in \cref{sec:cusps,sec:elliptic_points}.  However, this only gives a \emph{minimal resolution} for each singular point. There may still be curves to blow down to get the minimal model $X(\Gamma)$: see \cref{sec:Kodairatype}.

For these varieties we will also use the notation $X_0(\frakN)_\frakb=X(\Gamma_0(\frakN)_\frakb)$, and
\begin{equation}
    X_0(\frakN) \colonequals \bigsqcup_{[\frakb] \in \Cl^+ R} X_0(\frakN)_{\frakb}, \quad
    X_1(\frakN) \colonequals \bigsqcup_{[\frakb] \in \Cl^+ R} X_1(\frakN)_{\frakb},
\end{equation}
gathering components.  

\subsection{Hilbert modular forms}
\label{sub:hmf-def}

Hilbert modular varieties can be understood explicitly by studying their modular forms, defined as follows.  Let $\Gamma < \GL_2^+(F)$ be a congruence subgroup, and let $k=(k_i)_i \in 2\Z_{\geq 0}^n$. 
A \defi{Hilbert modular form} of (even) weight~$k$ for $\Gamma$ is a holomorphic function $f \colon \calH \to \C$ such that
\begin{equation} \label{eqn:fundfgammazprop}
f(\gamma z) = \left(\prod_v \frac{(c_v z_v + d_v)^{k_v}}{\det(\gamma_v)^{k_v/2}}\right) f(z)
\end{equation}
for all $\gamma \in \Gamma$ and $z\in \calH$, with the additional condition that $f$ remains bounded in vertical strips in the case $F=\Q$.

The $\C$-vector space of Hilbert modular forms for $\Gamma$ of weight $k$ is finite-dimen\-sional and denoted $M_k(\Gamma)$. Restricting to parallel weights (i.e., all~$k_i$ are equal to some~$k\in\Z)$, multiplication gives 
\begin{equation} M(\Gamma) \colonequals \bigoplus_{k \in 2\Z_{\geq 0}} M_k(\Gamma) \end{equation}
the structure of a graded ring.

A \defi{cusp form} is a Hilbert modular form $f$ such that $f(z) \to 0$ as $z \to c$ for all cusps $c$ of~$\Gamma$.  Let $S_{k}(\Gamma) \subseteq M_k(\Gamma)$ denote the subspace of cusp forms. 
There is a natural integration pairing, called the \defi{Petersson inner product}, between $M_k(\Gamma)$ and $S_k(\Gamma)$ which provides an orthogonal decomposition
\begin{equation}
M_k(\Gamma) = E_k(\Gamma) \oplus S_k(\Gamma);
\end{equation}
we call $E_k(\Gamma)$ the \defi{Eisenstein subspace}. 

For standard congruence subgroups, one can define related spaces of modular forms by gathering components: we write
\begin{equation}
% \begin{aligned}
M_k(\Gamma_0(\frakN)) \colonequals \bigoplus_{[\frakb] \in \Cl^+ R} M_k(\Gamma_0(\frakN)_\frakb),\\
%S_k(\Gamma_0(\frakN)) &\colonequals \bigoplus_{[\frakb] \in \Cl^+ R} S_k(\Gamma_0(\frakN)_\frakb),\\
% M(\Gamma_0(\frakN)) &\colonequals \bigoplus_{k \in 2\Z_{\geq 0}} M_k(\Gamma_0(\frakN)),
% \end{aligned}
\end{equation}
and similarly define $S_k(\Gamma_0(\frakN))$,  $M_k(\Gamma_0^1(\frakN))$, $M(\Gamma_0(\frakN))$, etc.

\begin{remark}
More expansively, we may also be interested in subspaces of lifts, forms of non-parallel (and even non-paritious) weight, the decomposition into eigenforms (with attached $L$-functions, Galois representations, and motives), etc.  Here, we restrict our scope to those aspects that are relevant for basic surface invariants.
\end{remark}

\section{Cusps} \label{sec:cusps}

A fundamental invariant of a Hilbert modular variety is the number of its cusps, the orbits of $\PP^1(F)$ under the attached congruence subgroup $\Gamma$, arising in its compactification.
% : this number is equal to~$\dim E_k(\Gamma)$ for even~$k\geq 4$. 
We wish to enumerate these cusps explicitly. These cusps are singular points (except for modular curves); for certain Hilbert modular surfaces, the number of curves in the minimal resolution, as well as their intersection numbers, are pleasantly determined by Hirzebruch--Jung continued fractions.

\subsection{Cusp enumeration} \label{sec:cuspenum}

For the groups~$\Gamma_1(\frakN)_\frakb$, Dasgupta--Kakde~\cite{DasguptaKakde} give an explicit enumeration of cusps.  We now recall their notation and method \cite[Section 3]{DasguptaKakde}, and we generalize it to the cases of~$\Gamma_0(\frakN)_\frakb, \Gamma_0^{1}(\frakN)_\frakb$, and $\Gamma_1^{1}(\frakN)_\frakb$. In this section,~$F$ is any totally real field.

For $i=0$ or $1$, a cusp of $\Gamma_i({\frakN})_{\frakb}$ or $\Gamma_i^{1}({\frakN})_{\frakb}$ is represented by $(a:c) \in \PP^1(F)$.
 We define
 \begin{equation} \label{eqn:cusp_ideals}
 \begin{aligned}
 \fraks &= \fraks(a,c) \colonequals aR + c \frakb^{-1} \\
 \frakM &= \frakM(a,c) \colonequals \frakN + c(\frakb \fraks)^{-1}.
 \end{aligned}
 \end{equation}

\begin{lemma}
The ideal $\frakM$ is well-defined (independent of the chosen representative of the cusp) and integral.
\end{lemma}

\begin{proof}
Scaling $a,c$ by an element of $F^\times$ changes $\fraks$ but this cancels out in $\frakM$. Since $c \in \frakb \fraks$, the ideal $\frakM$ is integral.
\end{proof}

For a fractional ideal $\fraks$ and an integral ideal $\frakM$ of $R$, let $(\fraks/ \fraks \frakM)^{\times}$ be the set of generators of $\fraks/\fraks\frakM$ as an $(R / \frakM)$-module, and let
\begin{equation} \label{eqn:RsM}
    \calR_{\frakM}^{\fraks} \colonequals (\fraks / \fraks \frakM)^{\times} / R_{>0}^{\times}, \quad
     \calR_{\frakM}^{\fraks,1} \colonequals (\fraks / \fraks \frakM)^{\times} / R^{\times 2} 
\end{equation}
be the quotients of this set under multiplication by totally positive units and under multiplication by squares of units, respectively. Let
\begin{equation} \label{eqn:P1Nb}
    \calP_{1}(\frakN)_{\frakb} \colonequals \left \{  
    (\fraks, \frakM, (a, c)) : \frakN \subseteq \frakM \subseteq R, (a,c) \in \calR_{\frakM}^{\fraks} \times \calR_{\frakN/\frakM}^{\fraks \frakb \frakM}
    \right \}.
\end{equation}
where $\fraks$ is a fractional ideal of $F$ and $\frakM$ is an integral ideal.

To deal with level $\Gamma_0(\frakN)_{\frakb}$, we further denote 
\begin{equation}
    \calR_{\frakM, \fraks, \frakb, \frakN} \colonequals 
    \left( \calR_{\frakM}^{\fraks} \times 
    \calR_{\frakN / \frakM}^{\fraks \frakb \frakM} \right) / (R/\frakN)^{\times},
\end{equation}
where the action of $(R/\frakN)^{\times}$ on 
$\calR_{\frakM}^{\fraks} \times \calR_{\frakN/\frakM}^{\fraks \frakb \frakM}$ is given by
\begin{equation} \label{eqn:x_x_inverse}
    x \cdot (a,c) = (xa, x^{-1}c),
\end{equation}
and similarly define
\begin{equation} \label{eqn:P0Nb}
    \calP_{0}(\frakN)_{\frakb} \colonequals \left \{  
    (\fraks, \frakM, (a, c)) : 
    \frakN \subseteq \frakM \subseteq R, (a,c) \in \calR_{\frakM, \fraks, \frakb, \frakN}
    \right \}.
\end{equation}
We then have the following.

\begin{lemma} \label{lem:cusp_bijection}
    There are natural bijections 
    \begin{equation}
    \begin{aligned}
        \varphi_1 : \Gamma_1(\frakN)_{\frakb} \backslash (F^2 \setminus \{0\})
        \rightarrow \calP_{1}(\frakN)_{\frakb} \\
        \varphi_0 : \Gamma_0(\frakN)_{\frakb} \backslash (F^2 \setminus \{0\})
        \rightarrow \calP_{0}(\frakN)_{\frakb}
    \end{aligned}
    \end{equation}
    given by $(a,c) \mapsto (\fraks, \frakM, (\bar{a}, \bar{c}))$, where $\fraks = \fraks(a,c)$ and $\frakM = \frakM(a,c)$ are as in \eqref{eqn:cusp_ideals}.
\end{lemma}

\begin{proof}
    The map $\varphi_1$ is well-defined and bijective by Dasgupta--Kakde \cite[Lemma 3.6]{DasguptaKakde}. To see that $\varphi_0$ is well-defined, 
    let $\gamma = (\begin{smallmatrix} p & q \\ r & t \end{smallmatrix}) \in \Gamma_0(\frakN)_{\frakb}$, so that $p,t \in R,\ q \in \frakb^{-1},$ and $ r \in \frakb \frakN$.
    Then 
    \begin{equation}
    \begin{aligned}
        \fraks(pa+qc, ra+tc) &= (pa+qc)R + (ra+tc)\frakb^{-1} \\
        &= a(pR + r\frakb^{-1}) + c(qR+t\frakb^{-1}) 
        \subseteq aR + c\frakb^{-1} = \fraks(a,c).
    \end{aligned} 
    \end{equation}
    Using $\gamma^{-1}$, we obtain the reverse inclusion, so $\fraks(pa+qc, ra+tc) = \fraks(a,c)$. 
    Similarly,
    \begin{equation}
    \begin{aligned}
        \frakM(pa+qc, ra+tc) &= \frakN + (ra+tc)(\frakb \fraks )^{-1} \\
        &= \frakN + ra (a \frakb + cR )^{-1} + tc (\frakb \fraks)^{-1}\\
        &\subseteq \frakN + c(\frakb \fraks)^{-1} = \frakM(a,c),
    \end{aligned}
    \end{equation}
    and using $\gamma^{-1}$, we have $\frakM(pa+qc, ra+tc) = \frakM(a,c)$.
    
    Since $q \in \frakb^{-1}$ and $r \in \frakb \frakN$, we see that 
    $qc \in c \frakb^{-1} \subseteq \fraks \frakM$ and 
    $ra \in a \frakb \frakN \subseteq \fraks \frakb \frakN$, hence 
    $(\overline{pa+qc}, \overline{ra+tc}) = (\overline{pa}, \overline{tc})$, where bars indicate classes in~$\calR_{\frakM, \fraks, \frakb, \frakN}$.
    We also know that $pt-qr = \det(\gamma) \in R^{\times}_{>0}$. In particular, as $aqr \in \fraks \frakN \subseteq \fraks \frakM$, we see that 
    \begin{equation}
        \overline{a} = \overline{(pt-qr)a} = \overline{pta-qra}
        = \overline{pta}.
    \end{equation}
    Since $t$ maps to $(R/\frakN)^{\times}$, its action yields
    $t\cdot (\overline{pa}, \overline{tc}) = (\overline{pta}, \overline{c}) = (\overline{a}, \overline{c})$, showing that the pairs are equivalent under the action, and establishing that $\varphi_0$ is well-defined.

    We now show that $\varphi_0$ is bijective.    Surjectivity follows from the surjectivity of~$\varphi_1$.     For injectivity, suppose that 
    \begin{equation} \label{eqn:cusps_phi_0}
    \varphi_0(a,c) = \varphi_0(a',c') = (\fraks, \frakM, (\bar{a}, \bar{c})).
    \end{equation}
    From the third component of \eqref{eqn:cusps_phi_0}, there exist an element $d \in (R/\frakN)^{\times}$ and totally positive units $\eps_a, \eps_c \in R^{\times}_{>0}$ such that 
    \begin{equation}
    \begin{aligned}
        a' &\equiv d^{-1} \eps_a a \pmod{\fraks \frakM} \\        
        c' &\equiv d \eps_c c \quad\,\pmod{\fraks \frakb \frakN}.
    \end{aligned}
    \end{equation}
    Choose $\delta = (\begin{smallmatrix} p & q \\ r & t \end{smallmatrix}) \in \Gamma_0(\frakN)_{\frakb}$ such that $p \equiv d^{-1} \bmod \frakN$ and $t \equiv d \bmod \frakN$. 
    Acting on $(a',c')$ by $\delta$, we may suppose that $a' \equiv \eps_a a \psmod{\fraks \frakM}$ and $c' \equiv \eps_c c \psmod{\fraks \frakb \frakN}$.
    The result now follows from the injectivity of $\varphi_1$.    
\end{proof}

Analogously, we define $\calR^{1}_{\frakM, \fraks, \frakb, \frakN}$ and $\calP_{i}^{1}(\frakN)_{\frakb}$ for $i = 0,1$ to deal with levels $\Gamma_i^{1} (\frakN)_{\frakb}$, and obtain analogous bijections
\begin{equation}
\begin{aligned}
    \varphi^1_1 : \Gamma_1^1(\frakN)_{\frakb} \backslash (F^2 \setminus \{0\})
    \rightarrow \calP^1_{1}(\frakN)_{\frakb} \\
    \varphi^1_0 : \Gamma_0^1(\frakN)_{\frakb} \backslash (F^2 \setminus \{0\})
    \rightarrow \calP^1_{0}(\frakN)_{\frakb}
\end{aligned}
\end{equation}
due to the exact sequence $0 \to P\Gamma_i^{1}(\frakN)_{\frakb} \to P\Gamma_i(\frakN)_{\frakb} \to R^{\times}_{>0} / R^{\times 2} \to 0 $.
    
For each $\frakM \mid \frakN$, we denote by $\frakQ_i(\frakM, \frakN)$ the set of cusps $(a:c)$ of $X_i(\frakN)$ such that $\frakM(a,c) = \frakM$, and by $\calP_{i}(\frakM, \frakN)_{\frakb} \subseteq \calP_{i}(\frakN)_{\frakb}$ the set of tuples whose second coordinate is $\frakM$, for $i \in \{0,1\}$.
We denote by $Q_0(\frakM, \frakN)$ the quotient of the product of narrow ray class groups
$\Cl^+(\frakM) \times \Cl^+(\frakN / \frakM)$ by the subgroup generated by
\begin{equation}
    \left \{ \left([xR], [x^{-1}R] \right) : x \in R, \ xR + \frakN = R \right \}.
\end{equation}
Similarly, we denote by $Q_1(\frakM, \frakN)$ the quotient of $\Cl^+(\frakM) \times \Cl^+(\frakN / \frakM)$ by the subgroup generated by 
\begin{equation}
    \left \{ \left([xR], [x^{-1}R] \right) : x \in R, \ x \equiv 1 \psmod{\frakN} \right \}.
\end{equation}
We then have the following corollary.

\begin{corollary}
    For all $\frakM \mid \frakN$ and $i=0,1$ we have 
    $\#\frakQ_i(\frakM, \frakN) = \#Q_i(\frakM,\frakN)$. Therefore, the number of cusps of $X_i(\frakN)$ is $\sum_{\frakM \mid \frakN} \#Q_i(\frakM, \frakN)$. 
\end{corollary}

\begin{proof}
    For $i = 1$, this statement is proven by Dasgupta--Kakde \cite[Corollary 3.12]{DasguptaKakde}. We will prove it for $i = 0$.
    By \Cref{lem:cusp_bijection}, taking a quotient by the natural action of $F^{\times}$ on both sides, we obtain a natural bijection
    \begin{equation} \label{eqn:disjoint_union}
        \frakQ_0(\frakM, \frakN) \rightarrow 
        \bigsqcup_{\frakb \in \Cl^+ R} \calP_0(\frakM, \frakN)_{\frakb} / F^{\times}.
    \end{equation}
    There is a surjective map
    \begin{equation}
    \begin{aligned}
        \calP_{0}(\frakM, \frakN)_{\frakb} / F^{\times} &\to \Cl R \\
        (\fraks, (a,c)) &\mapsto [\fraks].
        \end{aligned}
    \end{equation}
    Let $\calU$ be the image of $R^{\times}$ mapped diagonally to 
    \[ \calV(\frakM,\frakN) \colonequals \bigl( \calR_{\frakM}^{R} \times \calR_{\frakN/\frakM}^{R} \bigr) / (R / \frakN)^{\times}, \] where the quotient is by the action $x \mapsto (x,x^{-1})$. Then the fiber over a point in the above map is a principal homogeneous space for $\calV(\frakM, \frakN) / \calU$, which is independent of $\frakb$. Hence
    \begin{equation}    \label{eqn:cusp_count_M}
        \#\frakQ_0(\frakM,\frakN) = (h^+ h) \cdot \# (\calV(\frakM, \frakN) / \calU).
    \end{equation}
    We can now conclude using the exact sequence
    \begin{equation}
        1 \rightarrow \calV(\frakM, \frakN) / \calU \rightarrow Q_0(\frakM, \frakN)
        \rightarrow \Cl^+ R \times \Cl R \rightarrow 1.
        \qedhere
    \end{equation}
\end{proof}

The above proof also implies the following corollary.

\begin{corollary}
    The number of cusps of $X_0(\frakN)$ is equal to
    \begin{equation}
        h^+ h \smash{\sum_{\frakM \mid \frakN}}
        \varphi_{>0} (\frakM + \frakN/\frakM),
    \end{equation}
    where
    \begin{equation}
        \varphi_{>0}(\frakM) = \# \bigl( (R/\frakM)^{\times} 
        / R_{>0}^{\times} \bigr).
    \end{equation}
\end{corollary}

\begin{proof}
    Consider the exact sequence
    \begin{equation}
        \left( R/\frakN \right)^{\times} \rightarrow
        \left( R/\frakM \right)^{\times} \times 
        \left( R/(\frakN/\frakM) \right)^{\times} \rightarrow
        \left( R/(\frakM + \frakN/\frakM) \right)^{\times}
        \rightarrow 1,
    \end{equation}
    with the maps $r \mapsto (r,r^{-1})$ and $(r,s) \mapsto rs$. 
    Recalling \eqref{eqn:cusp_count_M}, we see that 
    \begin{equation}
    \calV(\frakM, \frakN) \simeq  \left( R/(\frakM + \frakN/\frakM) \right)^{\times} / R_{>0}^{\times}, 
    \end{equation}and this isomorphism identifies $\calU$ with squares of units. Since these are already totally positive, we deduce that the action of $\calU$ is trivial. (This could also be observed directly using weak approximation.)  The result follows.
\end{proof}

We define analogously $\calP_i^{1}(\frakM,\frakN)$, $Q_i^{1}(\frakM,\frakN)$ and $\varphi^{1}(\frakM) = \# \left( (R/\frakM)^{\times} / R^{\times 2} \right)$. It follows that the number of cusps of $X_i^1(\frakN)$ is $\sum_{\frakM \mid \frakN} \# Q_i^{1}(\frakM, \frakN)$. In particular, the number of cusps of $X_0^{1}(\frakN)$ is given by $h^+ h \sum_{\frakM \mid \frakN} \varphi^{1} \left(\frakM + \frakN/\frakM \right)$.

\subsection{Explicit computation of cusps}

In our implementation, we make two modifications to the enumeration method of \cref{sec:cuspenum} for computational convenience.
First, given $\fraks$ and~$\frakM$, we choose a generator $g$ for $\fraks/\fraks\frakM$ as an ($R/\frakM$)-module and work via the isomorphism 
\begin{equation} \label{eqn:RmodMiso}
    \begin{aligned}
        R/\frakM &\overset{\sim}{\to} \fraks/\fraks \frakM \\
        r &\mapsto g r.
    \end{aligned}
\end{equation}
We compute such a generator $g$ by ensuring that $g \not\in \fraks\frakp$ for all prime ideals $\frakp \mid \frakM$, which is possible by weak approximation. 
 Second, instead of forming the sets $\calR_{\frakM}^{\fraks}$ and $\calR_{\frakM}^{\fraks,1}$ (involving the quotient by the action of~$R_{>0}^{\times}$ or~$R^{\times2}$) and then further taking the quotient by $F^\times$ as in \cref{eqn:disjoint_union}, we consider only representatives of $\Cl R$ as ideals~$\fraks$ (thus accounting for the action of all units in~$F^\times$ except for those in~$R^\times$), and compute the remaining quotients all at once.
 
 We are thus led to the following algorithm for computing $\calP_{i}(\frakM, \frakN)_{\frakb}/F^\times$, where $i\in\{0,1\}$. We use the following notation:~$\eps$ is a fundamental unit of~$R$, and~$\eps^+$ is a fundamental totally positive unit.
 
\begin{enumerate}[1.]
    \item
        Form the direct product $D = (R/\frakM) \times (R/(\frakN/\frakM))$.
    \item Split according to~$i$.
    %\begin{enumerate} 
    %        \item [a.]
                If~$i=1$, quotient $D^\times$ by the diagonal action of~$R^\times_{> 0}$ in each coordinate, as well as the diagonal action of $R^\times$. (Quotienting by~$R^\times$ partially deals with the action of $F^\times$ mentioned above.) Explicitly, let
                $$
                H \colonequals \langle (\eps, \eps), (-1, -1), (\eps^+, 1), (1, \eps^+) \rangle \subseteq D^\times,
                $$
                and compute a transversal (i.e., a complete set of coset representatives) $T$ for $H$ in $D^\times$. 
    %        \item [b.]
                If~$i=0$, quotient $D^\times$ by same actions as above, as well as the action by $(R/\frakN)^\times$ given in \cref{eqn:x_x_inverse}. Explicitly, let $H$ be the subgroup of~$D^\times$ generated by
                $$
                \qquad\qquad\{(\eps, \eps), (-1, -1), (\eps^+, 1), (1, \eps^+)\}
                \cup \{(r, r^{-1}) : \text{$r$ is a generator of $(R/\frakN)^\times$} \}
                $$
                and compute a transversal $T$ for $H$ in $D^\times$.
    %            \end{enumerate}
    \item
        For each $[\fraks] \in \Cl R$, compute generators $g_1,g_2$ for $\fraks/\fraks \frakM$ and $\fraks \frakb \frakM/\fraks \frakb \frakN$ as described above. Let
                $$
                    Q_{\fraks, \frakM} \colonequals \{(\fraks, \frakM, (g_1 t_1, g_2t_2)) : (t_1, t_2) \in T\} \, .
                $$
        \item
        Return $\calP_{i}(\frakN)_{\frakb}/F^\times$ as
        $\displaystyle
             \bigcup_{[\fraks] \in \Cl R} \,\bigcup_{\frakM \mid \frakN} 
             Q_{\fraks, \frakM} \, .
        $
\end{enumerate}

We find the cusp corresponding to a tuple $(\fraks, \frakM, (\overline{a}, \overline{c})) \in \calP_{i}(\frakN)_{\frakb}/F^\times$ by inverting the bijections of \Cref{lem:cusp_bijection}. To do so, we seek $a,c \in R$ that are congruent to~$\overline{a}, \overline{c}$ modulo $\fraks \frakM$ and $\fraks \frakb \frakN$, respectively, such that $\gcd(c, \fraks \frakb \frakN) = \fraks \frakb \frakM$ and $\gcd(a, c \frakb^{-1}) = \fraks$. We first arbitrarily lift $\overline{a}$
%\in \frac{\fraks}{\fraks \frakM}$
and $\overline{c}$
%\in \frac{\fraks \frakb \frakM}{\fraks \frakb \frakN}$
to $a_0 \in \fraks$ and $c_0 \in \fraks \frakb \frakM$. Adding an element of~$\fraks \frakb \frakN$ to $c_0$ will not change $\gcd(c_0, \fraks \frakb \frakN)$, so we may take $c = c_0$. Let
\begin{equation}
B \colonequals \{\frakp \text{ prime ideal of }R : \frakp \mid c \frakb^{-1} \text{ and } \frakp \nmid \fraks \frakM\}.
\end{equation}
To find $a$, we construct an element $x \in R$ such that
\begin{itemize}
    \item $x \in \fraks \frakM$,
    \item if $\frakp \in B$ and $\frakp \mid a_0$, then $v_\frakp(x) = v_\frakp(\fraks)$, and
    \item if $\frakp \in B$ and $\frakp \nmid a_0$, then $v_\frakp(x) > v_\frakp(\fraks)$,
\end{itemize}
and then set $a = a_0 + x$. (This again amounts to specifying valuations of~$x$ at finitely many primes.) The first condition ensures that $a$ has the same reduction modulo $\fraks \frakM$ as $a_0$, and the latter two conditions guarantee that $\gcd(a, c \frakb^{-1}) = \fraks$.

We compute $\calP_{i}^{1}(\frakM, \frakN)_{\frakb} / F^{\times}$ using the above algorithm, replacing $\eps^+$ by $\eps^2$, and find the cusps corresponding to tuples in $\calP_{i}^1(\frakN)_{\frakb} / F^{\times}$ in the same manner.

\subsection{Resolving cusps} \label{sec:hjcont}

Cusps usually are very singular points of the Baily--Borel compactification $\overline{Y}(\Gamma)$.
In this section, we explicitly resolve cusps following van der Geer \cite[chapter II]{vdG} and van der Geer--Hirzebruch \cite[chapter II.1]{vdg-Hirzebruch}) assuming that $F$ is a real quadratic field. Hirzebruch's method readily applies to full level (principal congruence) subgroups~$\Gamma(\frakN)_{\frakb}$, and we explain how to generalize it to other natural congruence subgroups.
In the end, we obtain a description of the preimage of each cusp in the minimal desingularization $X(\Gamma) \to \overline{Y}(\Gamma)$ as a cyclic configuration of $\mathbb{P}^1$'s with known self-intersection numbers. This data is essential to the computation of the geometric invariants of the surface in~\cref{sec:geom_invs}.

We start by describing the general procedure for desingularizing cusps for any congruence group $\Gamma < \GL_2^+(F)$.
Let $(a : c) \in \mathbb{P}^1(F)$ be a cusp of $Y(\Gamma)$. There exists a matrix $\gamma \in \GL_2^+(F)$ taking $(a : c)$ to $\infty = (1 : 0)$. The resolution of $(a:c)$ can then be described in terms of the isotropy group of $\gamma\Gamma \gamma^{-1}$ at~$\infty$, in other words, the subgroup of upper-triangular matrices in $\gamma \Gamma \gamma\inv$, seen as a transformation group on $\calH$. Hirzebruch's method directly applies when this transformation group can be written as
\begin{equation} \label{eqn:GMV}
G(M, V) \colonequals \begin{pmatrix} V & M \\ 0 & 1\end{pmatrix},
\end{equation}
where $V \subseteq F^\times_{>0}$ is a group of totally positive units and $M\subset F$ is a $\Z$-module of rank~$2$ such that $VM=M$. We say that the cusp $(a : c)$ is of \defi{type} $G(M, V)$.

This condition is known to hold when~$\Gamma= \Gamma(\frakN)_{\frakb}$; we will see that it also holds when~$\Gamma$ is one of~$\Gamma_0(\frakN)_{\frakb}$, $\Gamma_0^1(\frakN)_{\frakb}$, $\Gamma_1^1(\frakN)_{\frakb}$, and also~$\Gamma_1(\frakN)_{\frakb}$ if~$\frakN$ is squarefree. 

\begin{remark}
For a general congruence subgroup~$\Gamma$ such that $\rmP\!\Gamma(\frakN)_{\frakb} \leq \rmP\!\Gamma$, a promising computational approach is to resolve the cusps for the action of~$\Gamma(\frakN)_{\frakb}$ and quotient the resulting smooth surface by the finite group $\rmP\!\Gamma_\frakb/\rmP\!\Gamma(\frakN)_{\frakb}$; we do not pursue this further in this paper.
\end{remark}

Once~$G(M,V)$ is known, the intersection numbers we are looking for can be explicitly computed as follows.  Let $v, v^{\prime} \colon F \hookrightarrow \bbR$ be the two ordered embeddings of $F$ in $\bbR$. We say that a $\bbZ$-basis~$(x, y)$ of $M$ is \defi{oriented} if 
\begin{equation}
v(x) v^{\prime}(y) - v^{\prime}(x) v(y) > 0.
\end{equation}
Necessarily, exactly one of $(x, y)$ or $(y, x)$ is oriented. Given an oriented basis $(x,y)$ of~$M$, we compute the \defi{Hirzebruch--Jung continued fraction} of $v(x/y)$, as follows:
\begin{equation}
v(x/y) = b_0 - \frac{1}{b_1 - \displaystyle\frac{1}{b_2 - \dots}} =: [[b_0, b_1, \dots]],
\end{equation}
where $b_0$ is the smallest integer greater than $v(x/y)$, then $b_1$ is the smallest integer greater than~$1/(v(x/y) - b_0)$, and so forth. This continued fraction is periodic; let~$d$ be its period. For each~$1\leq j\leq d$, we define
$$
w_j \colonequals [[b_j,b_{j+1},\ldots]].
$$
Let~$k\geq 1$ be minimal such that $(w_1\cdots w_d)^k\in V$. Then the resolution of the cusp~$(a:c)$ is a cyclic configuration of~$dk$
curves isomorphic to $\bbP^1$, with intersection numbers
$
(-b_0,\ldots,-b_{d-1})$ repeated~$k$ times,
except when~$dk=1$, in which case we find one rational curve with an ordinary double point and self-intersection $-b_0+2$ \cite[Chapter II, Lemma 3.2]{vdG}.

Our new contribution in this section is the explicit description of the type of a cusp~$(a:c)$ for standard congruence subgroups. Without loss of generality, we can assume~$\frakb$ and~$\frakN$ to be coprime. We also normalize our cusp representatives $(a : c)$ as follows: we assume that $a\in R,\, c \in \frakb$, and that the (now integral) ideal $\fraks = aR + c\frakb^{-1}$ is coprime to~$\frakN$.

\begin{lemma}
\label{lem:cusp-change}
    Let~$\fraks = a R + c\frakb^{-1}$.
    Then there exists $\lambda\in \fraks\inv$ and~$\mu\in \fraks\inv\frakb\inv$ such that
    \begin{equation} 
    \gamma = \left(\begin{matrix} \lambda & \mu\\ -c&a \end{matrix}\right) \in
    \left(\begin{matrix} \fraks^{-1} & \fraks^{-1} \frakb^{-1} \\ \fraks \frakb & \fraks \end{matrix}\right) \cap \GL^+_2(F)
    \end{equation}
    satisfies~$\det(\gamma)=1$ and $\gamma(a:c) = \infty$.
\end{lemma}
\begin{proof}
    Following van der Geer \cite[Proposition I.1.1]{vdG}, write
    \begin{equation}1 \in \fraks^{-1}\fraks = a \,\fraks^{-1} + c \,\fraks^{-1} \frakb^{-1}.\end{equation}
    Therefore there exist~$\lambda\in \fraks^{-1}$ and~$\mu\in \fraks^{-1}\frakb^{-1}$ such that~$\lambda a+\mu c = 1$.
\end{proof}

We now describe the types $G(M, V)$ of each cusp  for the above level subgroups.

\begin{prop}
    \label{prop:explicit-MV}
    Let~$\fraks = aR + c\frakb^{-1}$ as above, and let~$\frakM = \frakM(a,c) = \frakN + c (\frakb\fraks)^{-1}$. Then
    \begin{enumalph}
    
    \item For level~$\Gamma_0(\frakN)_{\frakb}$, the cusp~$(a:c)$ is of type~$G(M,V)$ where
    \begin{align*}
        M &= \fraks^{-2} \frakb^{-1} \frakN\, (\frakN + \frakM^2)^{-1}
        %\fraks^{-2} \frakb^{-1} \prod_{i\in I} \frakp_i^{\max\{e_i-2f_i, 0\}}
        , \text{ and}\\
        V &= \bigl\{v \in R^\times_{>0}: v\equiv 1 \psmod{ 
        \frakM + \frakM/\frakN}
        %\prod_{i\in I} \frakp_i^{\min\{f_i, e_i-f_i\}}
        \bigr\}.
    \end{align*}
    
    \item For level~$\Gamma_0^1(\frakN)_{\frakb}$, the cusp~$(a:c)$ is of type~$G(M,V)$ where
    \begin{align*}
        \quad\qquad M &= \fraks^{-2} \frakb^{-1} \frakN\, (\frakN + \frakM^2)^{-1}
        %\fraks^{-2} \frakb^{-1} \prod_{i\in I} \frakp_i^{\max\{e_i-2f_i, 0\}}
        , \text{ and}\\
        V &= \bigl\{v^2: v\in R^\times\textup{ and  }v^2 \equiv 1 \psmod{
        \frakN + \frakN/\frakM}
        %\prod_{i\in I} \frakp_i^{\min\{f_i, e_i-f_i\}}
        \bigr\}.
    \end{align*}
    
    \item For level~$\Gamma_1^1(\frakN)_{\frakb}$, the cusp~$(a:c)$ is of type~$G(M,V)$ where
    \begin{align*}
        \qquad\qquad\qquad M &= \fraks^{-2} \frakb^{-1} (\frakN/\frakM)
        %\fraks^{-2} \frakb^{-1} \prod_{i\in I} \frakp_i^{e_i - f_i}
        , \text{ and}\\
        V &= \bigl\{v^2: v\in R^\times\textup{ and  }v \equiv 1 \psmod{
        \frakN\, \frakM\, (\frakN + \frakM^2)^{-1}}
        %\prod_{i\in I} \frakp_i^{\max\{f_i,e_i-f_i\}}
        \bigr\}.
    \end{align*}
    
    \item Assume that~$\frakN$ is squarefree. 
    Let~$U\subset R^\times$ be the subgroup of units congruent to~$1$ modulo $\frakN/\frakM$. Then for level~$\Gamma_1(\frakN)_{\frakb}$, the cusp~$(a:c)$ is of type~$G(M,V)$ where
    \begin{align*}
        \qquad\qquad M &= \fraks^{-2} \frakb^{-1} (\frakN/\frakM), \text{ and}\\
        V &= \bigl\{v\in R^\times_{>0}: v \equiv \varepsilon \psmod{ \frakM} \textup{ for some }\varepsilon\in U
        %\prod_{i\in I} \frakp_i^{f_i}
        \bigr\}.
    \end{align*}
    \end{enumalph}
    In each case, we can compute a matrix~$\gamma\in \GL_2^+(F)$ sending~$(a:c)$ to~$\infty$ such that
    \begin{equation}
        \gamma^{-1}  
        \left(\begin{matrix} v&m\\0&1 \end{matrix}\right) \gamma %\in \mathrm{Stab}((a:c))
    \end{equation}
    lies in the chosen congruence subgroup of~$\Gamma_\frakb$ (as a transformation, i.e.~up to scalars) precisely when~$(v,m)\in V\times M$.
\end{prop}

\begin{proof}
    Let~$\gamma = 
    \left(\begin{smallmatrix} \lambda & \mu \\ -c & a \end{smallmatrix}\right)$ be as in \Cref{lem:cusp-change}. As transformations of~$\calH$, elements in the stabilizer of~$(a:c)$ take the form
    \begin{equation}
        N = \gamma^{-1}  
        \left(\begin{matrix} v&m\\0&1 \end{matrix}\right)
        \gamma = \left(\begin{matrix} 
        1 + a\bigl(\lambda(v-1) - c m\bigr) &
        a\bigl(\mu(v-1) + a m\bigr) \\
        c\bigl(\lambda(v-1) - c m\bigr) &
        v - a\bigl(\lambda(v-1) - c m\bigr)
        \end{matrix}\right)
    \end{equation}
    for some~$v\in F^\times_{>0}$ and~$m\in F$. Then~$N$ lies in~$\Gamma(1)_{\frakb}$ if and only if~$v\in R^\times_{>0}$ and $m\in \fraks^{-2}\frakb^{-1}$.  Note that by multiplying~$\gamma$ on the left by a matrix of the form
    \begin{equation} \left(\begin{matrix} 1&x\\0&1 \end{matrix}\right), \end{equation}
    we can make an affine change of variables of the form~$m\mapsto m+x(v-1)$ for any~$x\in F$, provided that $x(v-1)$ remains in $s^{-2}\frakb^{-1}$.
    
    Factor~$\frakN = \prod_{i\in I} \frakp_i^{e_i}$, where each~$\frakp_i$ is a prime of~$R$, and write $\frakM = \prod_{i\in I} \frakp_i^{f_i}$, where $0\leq f_i \leq e_i$ for all~$i$.
    The exponents~$f_i$ are related to the factorization of~$c$: we can write
    \begin{equation}
    c = c' \prod_{i\in I} \frakp_i^{f_i}
    \end{equation}
    where~$c'$ is prime to~$\frakp_i$ whenever~$f_i < e_i$.
    We now separate the four cases.

    \begin{enumalph}
    \item \emph{Level~$\Gamma_0(\frakN)_{\frakb}$.} We have~$N\in \Gamma_0(\frakN)_{\frakb}$ if and only if for all~$i\in I$, we have
    \begin{equation}\lambda(v-1) - c m \equiv 0 \pmod{\frakp_i^{e_i-f_i}}.\end{equation}
    If~$c$ is invertible modulo~$\frakp_i$, then after an affine change of variables as above, this relation becomes~$m \equiv 0 \pmod{\frakp_i^{e_i}}$. Otherwise,~$\lambda$ is invertible modulo~$\frakp_i$ (because~$\det g$ is a unit, and~$\fraks$ and~$\frakb$ are coprime to~$\frakN$). A necessary condition is that~$v \equiv 1\pmod{\frakp_i^{\min\{f_i, e_i-f_i\}}}$. Then another change of variables brings this relation into the form~$m \equiv0 \pmod{\frakp_i^{\max\{e_i-2f_i, 0\}}}$.

    \item \emph{Level~$\Gamma_0^1(\frakN)_{\frakb}$.} Compared to the case of~$\Gamma_0(\frakN)_{\frakb}$, we only have to add the condition that~$v$ is a square.

    \item \emph{Level~$\Gamma_1^1(\frakN)_{\frakb}$.} We have $N\in \Gamma_1^1(\frakN)_{\frakb}$ up to multiplication by a scalar matrix if and only if the following conditions hold for all $i\in I$:
    \smallskip
        \begin{itemize}
            \setlength{\itemsep}{3pt}
            \item $v=w^2$ is a square,
            \item $\lambda(w-w^{-1}) - c m \equiv 0 \pmod{\frakp_i^{e_i-f_i}}$, and
            \item $w^{-1} + a\bigl(\lambda(w-w^{-1}) - c m\bigr) = w + c\bigl(\mu(w^{-1}-w) - am\bigr)$ is equal to 1 modulo~$\frakp_i^{e_i}$.
        \end{itemize}
    \smallskip
    \noindent
    This implies a third relation
    \begin{equation}
    w - a\bigl(\lambda(w-w^{-1})-cm\bigr) \equiv 1 \pmod{\frakp_i^{e_i}}
    \end{equation}
    which can also be deduced from the determinant 1 condition.
    If~$f_i=0$, we can rewrite this system as $m\equiv 0 \pmod{\frakp_i^{e_i}}$ and $w\equiv 1 \pmod{\frakp_i^{e_i}}$. Otherwise, the second equation gives~$w\equiv 1\pmod{\frakp_i^{f_i}}$. Write~$w = 1+\eta$ where the valuation of~$\eta$ at~$\frakp_i$ is at least~$f_i$. Summing the second and third relations, we see that the valuation of~$\eta$ must be at least~$e_i/2$. The equations become:
    \begin{equation}2\lambda \eta - c m \equiv 0 \pmod{\frakp_i^{e_i - f_i}}, \quad
            a\bigl(2\lambda\eta - c m\bigr) \equiv \eta \pmod{\frakp_i^{e_i}}.
    \end{equation}
    Since~$a$ is invertible modulo~$\frakp_i$, the valuation of~$2\eta$ must be at least~$e_i - f_i$. (Note that $\max\{e_i-f_i, f_i\}\geq e_i/2$ always.) The final equation can also be rewritten as $(1+2\mu c)\eta = ac m \pmod{\frakp_i^{e_i}}$. Since the valuation of~$\eta$ is at least~$f_i$, an affine change of variables brings this relation into the form~$m = 0 \pmod{\frakp_i^{e_i-f_i}}$.
    
    \item \emph{Level~$\Gamma_1(\frakN)_{\frakb}$.} In general, we have~$N\in \Gamma_1(\frakN)_{\frakb}$ up to multiplication by a scalar matrix if and only if there exists a unit~$\varepsilon\in R^\times$ such that the following conditions hold for all~$i\in I$:
    \begin{equation}
    \begin{aligned}
        \lambda(v-1) - c m &\equiv 0 \pmod{\frakp_i^{e_i-f_i}}, \text{ and}\\
        1 + a\bigl(\lambda(v-1) - c m\bigr) &= v + c\bigl(\mu(1-v) - am\bigr) \equiv \varepsilon \pmod{\frakp_i^{e_i}}.
    \end{aligned}
    \end{equation}
    Assuming~$\frakN$ squarefree, the only cases are~$f_i=0$ and~$f_i=1$ (both with~$e_i=1$). In the first case, after a change of variables, the system becomes~$m\equiv 0\pmod {\frakp_i}$ and~$\varepsilon\equiv 1 \pmod{\frakp_i}$, thus~$\varepsilon\in U$. In the second case, the system reads~$v\equiv \varepsilon \pmod{\frakp_i}$.
    \end{enumalph}

    The required forms of~$G(M,V)$ are then obtain from recombining the local conditions.
    In each case, a suitable global change of variables $m\mapsto m+x(v-1)$ can be  computed using the Sun Zi theorem (CRT) to obtain the desired matrix~$\gamma$.
\end{proof}

\section{Elliptic points} \label{sec:elliptic_points}

In general, a congruence subgroup will not act freely on $\calH$. In this section, we seek to understand algorithmically the fixed points of the action.  The key idea will be to relate the number of elliptic points with optimal embedding numbers of quadratic orders \cite[Chapters 30 and 39]{Voight}.

\subsection{Setup}
\label{sec:elliptic_pts_setup}

We refer to van der Geer \cite[\S I.5]{vdG} for basic facts about elliptic points.  Let $\Gamma \leq \GL_2^+(F)$ be a standard congruence subgroup \eqref{eqn:gammagroups} and let $\calO$ be the $R$-order in $\M_2(F)$ generated by $\Gamma$.  

Recall that the action of $\Gamma$ on $\calH$ factors through the faithful action of the group $\rmP\!\Gamma \leq \PGL_2^+(F)$.  We will understand stabilizers in the quotient modulo scalars, but will then want to lift them back to $\GL_2^+(F)$.  

If $\overline{\gamma} \in \rmP\!\Gamma \smallsetminus \{1\}$ has a fixed point $z \in \calH$, then this point is isolated and $\tr(\overline{\gamma})^2 - 4 \det(\overline{\gamma})$ is (well-defined up to squares and) totally negative; in this case, we call $\overline{\gamma}$ \defi{elliptic} and its fixed point an \defi{elliptic point} for $\Gamma$ (or $\rmP\!\Gamma$).  For any $z \in \calH$, the stabilizer $\Stab_{\rmP\!\Gamma}(z) \colonequals \{\gamma \in \rmP\!\Gamma : \gamma z = z\}$ is a finite cyclic group.  When this cyclic group is nontrivial generated by $\overline{\gamma}$, we obtain a quadratic CM (totally imaginary) extension $K \colonequals F(\gamma)$.  We descend this to an order as follows.

Let $z \in \calH$ be an elliptic point, and let $\gamma \in \Gamma$ be such that the image $\overline{\gamma} \in \rmP\!\Gamma$ generates the (nontrivial) cyclic group $\Stab_{\rmP\!\Gamma}(z)$.  Let $q$ be the order of $\overline{\gamma}$.  Rescaling by a unit, we may suppose without loss of generality that $\det(\gamma) \in R_{>0}^{\times}$ belongs to a fixed set of representatives for $R_{>0}^\times/R_{>0}^{\times 2}$ including $1$.

\begin{lemma} \label{lem:possibleEllipticOrders}
The following statements hold.
\begin{enumalph}
\item If $\det(\gamma)=1$, then $R[\gamma] \iso R[\zeta]$ for $\zeta$ a primitive root of unity of order $2q$ such that $\zeta + \zeta^{-1} \in F$; if $[F:\Q]=2$, then $1 < q \leq 6$. 
\item 
If $\det(\gamma)=u \not\in R_{>0}^{\times 2}$, then $2 \mid q$ and $\gamma^2=u^{-1}\zeta$ where $\zeta$ is a root of unity of order $q$ with again $\zeta+\zeta^{-1} \in F$; if $[F:\Q]=2$, then $1 \leq q \leq 6$.  
\end{enumalph}
\end{lemma}

\begin{proof}
For part (a), since $\overline{\gamma}^q=1$ and $\det(\gamma)=1$, we have $\gamma^q \in F^\times$ so $\gamma^q =\pm 1$ since $F$ is totally real.  If $q$ is even, then since $\overline{\gamma}$ has order $q$ we must have $\gamma^q=-1$ so $\gamma^{2q}=1$; otherwise, $q$ is odd and either $\gamma^{q}=-1$ already or $(-\gamma q)$ has order $2q$.  In either case, $R[\gamma] \simeq R[\zeta]$ and $\tr(\gamma)=\zeta+\zeta^{-1} \in F$.  The second statement follows from enumeration of cyclotomic polynomials.  
For part (b), we repeat (a) with $\gamma^2$ in place of $\gamma$.
See also Prestel \cite[\S 8]{Prestel} and van der Geer \cite[p.~16]{vdG}.
\end{proof}

The collection of minimal polynomials for the elements $\gamma$ from \Cref{lem:possibleEllipticOrders} is finite and effectively computable.  We let $\Omega_q$ denote the isomorphism classes of $R$-orders $S$ such that $S \iso R[\gamma]$ and $\langle \gamma \rangle = S^\times/R^\times$ has order $q$.

The assignment $\langle \overline{\gamma} \rangle \mapsto R[\gamma]$ described above yields a bijection from the set of $\Gamma$-conjugacy classes of elliptic points to the $\Gamma$-conjugacy classes of pairs $\phi, \overline{\phi}\colon S \hookrightarrow \calO$, where $S$ is a CM $R$-order with $S^\times > R^\times$ and $\overline{\phi}$ denotes the precomposition by the nontrivial involution of $S$.  

Thus, one can compute elliptic points in two steps. First, determine the list of $R$-orders $R[\gamma]$ such that
$\gamma$ generates a quadratic CM extension. Second, for each $S$ in the aforementioned list, determine the number of embeddings of $S$
into $S$ up to $\Gamma$-conjugacy. The first task is explained by \Cref{lem:possibleEllipticOrders}.

We now turn to counting embeddings of a given order.
Let $K$ be a quadratic CM extension of $F$ with ring of integers~$\Z_K$, and let $S \subset K$ be an order of $K$. We also introduce the following notation: $\calO^1$ and $\calO^\times_{>0}$ denote the subgroups of~$\calO^\times$ consisting of matrices~$\gamma$ such that $\det(\gamma) = 1$ or~$\det(\gamma)\in F^\times_{>0}$, respectively.

We say that an embedding $\phi\colon S \hookrightarrow \calO$ is \defi{optimal} if $\phi(K) \cap \calO = \phi(S)$. Not all embeddings of $S$ are necessarily optimal; however, given an embedding $\phi$, there exists a unique order $S \subseteq S' \subset K$ such that $S'$ is optimally embedded under $\phi$, namely $S' \colonequals \phi^{-1}(\phi(K) \cap \calO)$. We then obtain the decomposition
\begin{equation}
    \{\text{embeddings of $S$ into $\calO$}\} = 
    \bigsqcup_{S \subseteq S' \subset K} \{\text{optimal embeddings $S' \hookrightarrow \calO$}\}.
\end{equation}
It is well-known how to enumerate such superorders $S'$. Let $\frakf_S$ be the conductor of~$S$. Then for every divisor $\frakd \mid \frakf_S$, there exists an order $S' \subset K$ containing $S$ with conductor $\frakd$ (as a $\Z_K$-module, so $\Z_K$ has trivial conductor).

For a group $\calO^1 \subset \Gamma \subset \calO^\times$, we denote by $m(S, \calO; \Gamma)$ the number of optimal embeddings of $S$ into $\calO$ up to $\Gamma$-conjugacy. The previous decomposition shows that
\begin{equation}
    \#\{\text{$\Gamma$-conjugacy classes of $\phi\colon S \hookrightarrow \calO$}\} 
    = \sum_{S \subseteq S' \subset K} m(S', \calO; \Gamma).
\end{equation}
In order to compute the numbers $m(S', \calO; \Gamma)$ we use a local-global principle. 
% The property of being an optimal embedding is local, so the numbers $m(S', \calO; \Gamma)$ can be computed adelically.
The field $K$ does not satisfy the optimal selectivity condition \cite[31.1.6, condition (a) of Proposition 31.2.1]{Voight} since $B=\M_2(F)$ is split at all real places.
Let $\Zhat \colonequals \smash{\underleftarrow{\lim}_n} \Z / n \Z$ be the profinite completion of $\Z$, let $\Shat' \colonequals S' \otimes \Zhat$ and $\calOhat \colonequals \calO \otimes \Zhat$. 
Thus for every $R$-order $S' \subseteq K$ we have
    \begin{equation} \label{eqn:mso01xx}
        m(S',\calO; \calO^\times) = \frac{h(S)}{h(R)} m(\Shat',\calOhat;\calOhat^\times) ,
    \end{equation}
where $h(S) = \#\Pic S$ \cite[Corollary 31.1.10]{Voight} and $m(\Shat',\calOhat;\calOhat^\times)$ denotes the number of adelically optimal embeddings of the adelic order $\Shat'$, up to $\calOhat$-conjugacy.  
To recover $m(S', \calO; \Gamma)$, we use the formula \cite[Lemma~30.3.14]{Voight}
\begin{equation}
m(S',\calO; \Gamma) = m(S', \calO; \calO^\times) [\nrd(\calO^\times) : \nrd(\Gamma) \nrd(S'^\times)].
\end{equation}
The adelic embedding numbers $m(\Shat,\calOhat;\calOhat^\times)$, a product 
    \begin{equation}
        m(\Shat,\calOhat;\calOhat^\times) = \prod_{\frakp \mid \frakN} m_{\frakp}( S,\calO;\calO^\times)
    \end{equation}
of (finitely many) local embedding numbers which are given explicitly \cite[\S 30.6 and \S 30.7]{Voight}.

\subsection{Formula}\label{sec:ellipticPointsFormula} For simplicity, we concentrate on the case of the modular groups~$\Gamma_0(\frakN)_{\frakb}$ and~$\Gamma_0^1(\frakN)_{\frakb}$. By \eqref{eqn:g01bconj}, we may and do suppose that $\frakb$ is coprime to the order of any torsion point, i.e., coprime to $2,3,5$. Consider the order
\begin{equation}
\calO=\calO_0(\frakN)_\frakb=\begin{pmatrix} R & \frakb^{-1} \\ \frakN\frakb & R \end{pmatrix}
\end{equation} 
 and notice that $\calO^1 = \Gamma_0^1(\frakN)_\frakb$ and $\calO_{>0}^{\times} = \Gamma_0(\frakN)_\frakb$.

For $q \geq 2$, we denote by $m_q^1$ and~$m_q^+$ the number of elliptic points of order $q$ in $\calO^1\backslash\calH$ and~$\calO_{>0}^\times\backslash\calH$, respectively. From the previous section we have that
    \begin{equation}  \label{eqn:embed0q}
    \begin{aligned}
        m_q^1 & = \frac{1}{2} \sum_{\substack{S \supseteq R[\zeta_{2q}] \\ \#S_{\textup{tors}}^\times=2q}} m(S,\calO;\calO^{1}),\\
        m_q^+ & = \frac{1}{2} \sum_{S \in \Omega_q} \sum_{\substack{S' \supset S \\ \#S'^\times/R^\times = q}} m(S',\calO;\calO_{>0}^{\times}).
        \end{aligned}
    \end{equation}

\begin{prop} \label{prop:embedclassno}
    We have
    \begin{equation}
        m_q^1=\frac{2^{n-1}}{h(R)}\sum_{S} \frac{h(S)}{Q(S)}\, m(\Shat,\calOhat;\calOhat^\times) 
    \end{equation}
    and
    \begin{equation}
        m_q^+=\frac{2^{n-1}}{h^+(R)}\sum_{S} h(S)\, m(\Shat,\calOhat;\calOhat^\times), 
    \end{equation}
    where $S$ runs over orders as in \eqref{eqn:embed0q}
    and $Q(S)$ is the Hasse unit index.
\end{prop}

\begin{proof}
To compute $m_q^1$ we repeat the argument in Voight \cite[Proposition 39.4.12]{Voight}.  The only difference is that $\left[R^\times_{>_\Omega0} : R^{\times 2}\right]=\left[R^\times : R^{\times 2}\right] = 2^n$ because $\Omega=\emptyset$, i.e., there are no real ramified places in $B$.  For $m_q^+$, we repeat the same calculation with $\calO_{>0}^\times$, so 
\begin{equation}
    \left[\nrd\left(\calO_{>0}^\times\right)\nrd\left(S^\times\right) : R^{\times 2}\right]=\left[R^\times_{>0} : R^{\times 2}\right] = h^+(R)/h(R). \qedhere
\end{equation}
\end{proof}

Note that the number of elliptic points of a given order $q$ is independent of the component~$\frakb$. However, as we discuss in the next section, the rotation types may differ; consequently, their contributions to surface invariants may differ.

\subsection{Rotation factors} \label{sec:rotationfactors}

We follow Prestel \cite{Prestel}.  Let $z = (z_v)_v \in \calH$ be an elliptic point with stabilizer group $\langle \gamma \rangle \leq \Gamma_0(\frakN)_{\frakb}$ 
%as in \cref{sec:elliptic_pts_setup}
; let $\tr(\gamma) = t$ and $\det(\gamma) = u$, so that $\gamma^2-t\gamma+u=0$ and $t^2-4u$ is totally negative. 
If $\gamma = \left(\begin{smallmatrix} a & b \\ c & d \end{smallmatrix}\right)$, then
\begin{equation}
z_v = \frac{a_v - d_v}{2c_v} + \frac{1}{2|c_v|} \sqrt{t_v^2 - 4u_v}
\end{equation}
(roots in the upper half-plane).  The transformation ${z' \mapsto (z' - z)/(z' - \overline{z})}$ maps $z$ to $(0,\dots,0)$; the elliptic matrix $\gamma$ then acts as a rotation $z \mapsto \zeta z = (\zeta_v z_v)_v$ of the product of $n$ unit discs, rotating the $v$-component by 
\begin{equation}
\zeta_v = \left(\frac{t_v^2}{2u_v} - 1 \right) - i \sgn(t_v c_v) \sqrt{1 - \left(\frac{t_v^2}{2u_v} - 1 \right)^2},
\end{equation}
so each $\zeta_v$ is a primitive root of unity whose order matches $\gamma$.  We call the tuple $\zeta = (\zeta_v)_v$ the \defi{rotation factor} of $z$.

If $z'$ is another fixed point with stabilizer $\langle \gamma' \rangle$ with again $\tr(\gamma')=t$ and $\det(\gamma')=u$, then by the Skolem--Noether theorem, there exists $\beta \in \GL_2(F)$ such that $\gamma' = \beta^{-1}\gamma \beta$.  If $\gamma$ has rotation factor $\zeta$, then $\gamma'$ has rotation factor 
\begin{equation} \label{eqn:zetasign}
\zeta' = \zeta^{\sgn(\det(\beta))};
\end{equation}
in particular, $\zeta'=\zeta$ if and only if $\det(\beta)$ is totally positive.  

Connecting this up with the previous section, we need consider optimal embeddings but with attention to signs of the determinant---we will encounter a modified form of the selectivity phenomenon, so we follow the notation and conventions in Voight \cite[Chapter 31]{Voight}, some of which was already introduced in \cref{sec:elliptic_pts_setup}.  Let $K=F(\gamma) \supseteq F$ be a CM field.  Choose a fixed embedding $K \hookrightarrow \M_2(F)$ and identify $K$ with its image; this will serve as a reference point (like a base point of a fundamental group, see also below).  

\begin{definition}
An embedding $\phi_\beta \colon K \hookrightarrow \M_2(F)$ corresponding to conjugation by an element $\beta \in \GL_2(F)$ is \defi{oriented} if $\det(\beta) \in F_{>0}^\times$.
\end{definition}

Let $S \subset K$ be an $R$-order and let $\calO \subset \M_2(F)$ be an $R$-order.  An (optimal) embedding $\phi \colon S \hookrightarrow \calO$ determines an embedding $K \hookrightarrow \M_2(F)$, so we similarly define such an order to be \defi{oriented}.  Let $\Emb^+(S,\calO)$ be the set of oriented optimal embeddings $\phi \colon S \hookrightarrow \calO$ and let $\Emb^+(S,\calO;\calO_{>0}^{\times})$ be this set up to conjugation by~$\calO_{>0}^\times$.  Let 
\begin{equation} 
E^+ \colonequals \{\beta \in \GL_2^+(F) : \beta^{-1}K\beta \cap \calO = \beta^{-1} S \beta\}. 
\end{equation}
Then \cite[(30.3.13)]{Voight} the map which sends %the class of 
$\beta$ to conjugation by $\beta$ induces a bijection
\begin{equation} \label{eqn:Eplus}
K^\times \backslash E^+ / \calO_{>0}^{\times} \xrightarrow{\sim} \Emb^+(S,\calO;\calO_{>0}^{\times}).
\end{equation}
Let 
\begin{equation}
m^+(S,\calO;\calO_{>0}^{\times}) \colonequals \#\Emb^+(S,\calO;\calO_{>0}^{\times}).
\end{equation}

With this notation in hand, we can generalize the notion of selectivity to oriented optimal embeddings as follows.  Let $\Gen \calO$ be the genus of the order $\calO$, i.e.~the set of orders in $\M_2(F)$ which are locally isomorphic to~$\calO$.

\begin{definition}
We say that $\Gen \calO$ is \defi{orientedly genial} for $S$ if $\Emb^+(S,\calO') \neq \emptyset$ for all $\calO' \in \Gen \calO$; otherwise, we say that $\Gen \calO$ is \defi{orientedly optimally selective}.  
\end{definition} 

In other words, $\Gen \calO$ is orientedly genial for $S$ if and only if $S$ has an oriented embedding into every order $\calO'$ that is locally isomorphic to $\calO$.  In terms of rotation factors, if $\gamma$ has rotation factor $\zeta$ and $S \supseteq R[\gamma]$, then $\Gen \calO$ is orientedly genial for $S$ if and only if every order $\calO'$ locally isomorphic to $\calO$ admits an optimal embedding of $S$ with rotation factor $\zeta$.  

Denoting by $\Fhat \colonequals F \otimes \Zhat$ the finite adele ring of $F$, 
we attach \cite[(28.5.8)]{Voight}
\begin{equation}
GN^+(\calO) \colonequals F_{>0}^\times \det(N_{\GL_2(\Fhat)}(\calOhat)) \leq \Fhat^\times. 
\end{equation}
to the order $\calO$, as well as the class group 
\begin{equation} 
\Cl_{GN^+(\calO)} R \colonequals \Fhat^\times/GN^+(\calO). 
\end{equation}
The determinant map induces a bijection between $\Typ^+ \calO$, the similarly defined oriented type set of $\calO$, and $\Cl_{GN^+(\calO)} R$ \cite[Corollary 28.5.10]{Voight}.  By class field theory, attached to~$GN^+(\calO)$ is an abelian extension $H_{GN^+(\calO)} \supseteq F$.  

We define the \defi{orientedly optimal selectivity condition}:
\smallskip
\begin{enumerate}
\item[(OOS)] $K$ is a subfield of the class field $H_{GN^+(\calO)} \supseteq F$.
\end{enumerate}
\smallskip
We now restrict to the case at hand: we suppose that $\calO$ is an Eichler order of level~$\frakN$ \cite[\S 31.2]{Voight}.  Then there is a surjection $\Cl^+ R \to \Cl_{GN^+(\calO)} R$, presenting this class group as the quotient of the narrow class group by the primes $\frakp \mid \frakN$ with $\ord_\frakp(\frakN)$ odd. Thus (OOS) holds if and only if the following two conditions hold:
\begin{enumerate}
\item[(i)] $K$ is unramified at all nonarchimedean places $v \in \Pl F$; and 
\item[(ii)] If $\frakp \mid \frakN$ with $\ord_\frakp \frakN$ odd, then $\frakp$ splits in $K$.
\end{enumerate}

\begin{theorem}[Oriented optimal selectivity] \label{thm:oos}
Suppose that $\calO$ is an Eichler order.  Then the following statements hold.
\begin{enumalph}
\item $\Gen \calO$ is orientedly optimally selective for $S$ if and only if the orientedly optimal selectivity condition \textup{(OOS)} holds.
\item If $\Gen \calO$ is orientedly optimally selective for $S$, then $\Emb^+(S,\calO') \neq \emptyset$ for precisely \emph{half} of the types $[\calO'] \in \Typ^+ \calO$.  More precisely, if $[\calO'] \leftrightarrow [\frakb] \in \Cl^+ R$, then $\Emb^+(S,\calO') \neq \emptyset$ if and only if $\,\Frob_{\frakb} \in \Gal(K\,|\,F)$ is trivial.
\item In all cases, 
\begin{equation} m(S,\calO';\calO_{>0}^{\prime\times}) = m(S,\calO;\calO_{>0}^{\times}) \end{equation}
for all $\calO' \in \Gen \calO$ whenever both sides are nonzero.
\end{enumalph}
\end{theorem}

\begin{proof}
Amazingly, the proof given by Voight \cite[Main Theorem 31.1.7]{Voight} applies, \emph{mutatis mutandis} starting with \eqref{eqn:Eplus}.  
% Without introducing even more notation, we quickly summarize: the selectivity sandwich \cite[31.3.14]{Voight}, once oriented, reads
%\begin{equation}
%F_{>0}^{\times} \det(\widehat{K}^\times) \overset{\textup{(OOS)}}{\leq} \det(\widehat{K}^\times) GN^+(\calO) \overset{m}{\leq} F_{>0}^\times \det(\widehat{E}) \overset{s}{\leq} \widehat{F}^{\times}.
%\end{equation}
%Again, the first inequality (OOS) is an equality if and only if the orientedly optimal selective condition (OOS) holds, the middle inequality $m$ is always an equality (implying equality of oriented embedding numbers when nonzero), and the final inequality $s$ is an equality if and only if there is no oriented selectivity obstruction.  
The refinement in part (b) is also proven in the same way \cite[Proposition 31.4.4, Corollary 31.4.6]{Voight}.
\end{proof}

\begin{cor}
\!All rotation factors for $\gamma$ occur equally across the groups $\Gamma_0(\frakN)_\frakb$.  
\end{cor}

\begin{proof}
Apply \Cref{thm:oos}: condition (i) above fails since $K=F(\gamma) \supseteq F$ is CM so ramified at a nonarchimedean place $v \in \Pl F$.
\end{proof}

We now return to the issue of the choice of reference embedding $K \hookrightarrow \M_2(F)$.  

\begin{lemma} \label{lem: OOS conductor}
Assume that $\Gen \calO$ is orientedly optimally selective for $S$.  
Let $S \subseteq S_0$ be a suborder of conductor $\frakf = \frakf(S_0 | S)$, and assume that $\Emb^{+}(S_0,\calO) \ne \emptyset$. Then
$\Emb^{+}(S,\calO)$ is nonempty if and only if $\Frob_{\frakf} \in \Gal(K\,|\, F)$ is trivial.
\end{lemma}

\begin{proof}
Let $\widehat{S_0} \colonequals S_0 \otimes_{\Z} \Zhat$, and let $\widehat{\phi} \colon \widehat{S_0} \hookrightarrow \widehat{\calO}$ be a local orientedly optimal embedding, corresponding to $\widehat{\beta} \in \GL_2(\widehat{F})$ with the property that $\widehat{\beta}^{-1} \widehat{K} \widehat{\beta} \cap \widehat{\calO} = \widehat{\phi}(\widehat{S_0})$.  Then $\det(\widehat{\beta})=\widehat{f}$ where $\widehat{f} \widehat{R} \cap F = \mathfrak{f}$.  Exactly as in the proof of \Cref{thm:oos}(b), $\widehat{\beta}$ yields an orientedly optimal embedding $\phi \colon S_0 \hookrightarrow \mathcal{O}$ if and only if $\widehat{f} \in F_{>0}^\times \det(\widehat{K}^\times)$ if and only if $\Frob_{\frakf} \in \Gal(K\,|\,F)$ is trivial.  
\end{proof}

\begin{cor}
Let $K$ be a CM field, let $S \subset K$ be an $R$-order, and let $\gamma \in S^\times \smallsetminus R^\times$ be such that $\gamma R^\times \in S^\times/R^\times$ has finite order and generates $(S^\times/R^\times)_{\textup{tors}}$.  Let $f_\gamma(x)=x^2-t_\gamma x+u_\gamma \in R[x]$ be the minimal polynomial of $\gamma$.  Suppose that the condition (OOS) holds.  Then the following statements hold.
\begin{enumalph}
\item The order $R[\gamma]$ embeds orientedly optimally into
\begin{equation}\mathcal{O}_0(\frakN)_{(1)}=\begin{pmatrix} R & R \\ \frakN & R \end{pmatrix}
\end{equation}
if and only if $\widehat{R}[\gamma] \hookrightarrow \widehat{\calO}$ embeds optimally % (i.e., if and only if $m(\widehat{R}[\gamma],\widehat{\calO};\widehat{\calO}^\times) \neq 0$) 
if and only if there exists $x \in R$ such that $f_\gamma(x) \equiv 0 \pmod{\frakN}$, via the embedding
\begin{equation} \label{eqn:phix}
\phi(x)=\begin{pmatrix} x & 1 \\ -f_\gamma(x) & t_\gamma-x \end{pmatrix}.
\end{equation}
\item Let $x \in R$ be such that $f_\gamma(x) \equiv 0 \pmod{\frakN}$, and let $\zeta$ be the rotation type for 
\eqref{eqn:phix}.  Let $\frakf = \frakf(S\, |\, R[\gamma] )$ be the conductor of $R[\gamma] \subseteq S$, so $\disc R[\gamma] = \frakf^2 \disc S$.  Then the rotation factors $\zeta'=(\zeta_v^{\eps_v})_v$ which occur for fixed points of optimal embeddings of $S$ (with stabilizer of order $q$) into $\calO_0(\frakN)_{\frakb}$ are exactly those
with 
\begin{equation} \prod_v \eps_v = \displaystyle{\legen{K}{\frakf\frakb}}, \end{equation}
where $\displaystyle{\legen{K}{\frakf\frakb}} \in \{\pm 1\}$ is trivial if and only if $\Frob_{\frakf\frakb} \in \Gal(K\,|\,F)$ is trivial.
\end{enumalph}
\end{cor}

\begin{proof}
For (a), we choose the given rational canonical form $\phi \colon F(\gamma) \to M_2(F)$ as our reference point.  (For more on the local statement and normalized embeddings, we refer to Voight \cite[\S 30.6]{Voight}.)

For (b), starting with the reference point (a), we combine \Cref{thm:oos}(b) and \Cref{lem: OOS conductor} with the relationship to rotation factors given in \eqref{eqn:zetasign}: $S$ embeds orientedly optimally into $\calO_0(\frakN)_{\frakb}$ if and only if $\Frob_{\frakf\frakb}$ is trivial in $\Gal(K\,|\,F)$.  Considering now all possible orientations, we obtain exactly those with $\eps = (\eps_v)_v$ in the kernel of the composition of group homomorphisms
\begin{equation} 
\{\pm 1\}^n \to \frac{\Cl^+ R}{\Cl R} \to \Gal(K\,|\,F)
\end{equation}
giving the reformulation in (b).
\end{proof}

\begin{remark}
The results above generalize fully to the setup considered in Voight \cite[Chapter 31]{Voight}, allowing a quaternion algebra $B$ in place of $\M_2(F)$ over a global field $F$.
\end{remark}

\subsection{Resolution of singularities}

Now suppose $n=2$. Then van der Geer \cite[\S II.6]{vdG} explains how to resolve the singularities at the elliptic points, as cyclic quotient singularities.  The resolution of an elliptic point of order $q$ and rotation factor $(\zeta_q, \zeta_q^r)$ is given by constructing a (finite) Hirzebruch--Jung continued fraction expansion $[[b_1,\ldots b_d]]$ for $q/r$, as in \cref{sec:hjcont}.
Then there are $d$ curves $C_1, \ldots, C_d$ in the resolution chain, with self intersection numbers $C_1^2 = -b_1, \ldots, C_d^2 = -b_d$.
In addition $C_i \cdot C_j = \delta_{|i-j|,1}$ for $i \neq j$.
The local Chern class of this chain is
\begin{equation}
\sum_{i=1}^d (x_i + y_i - 1) C_i,
\end{equation}
the points $P_i = (x_i, y_i)$ being determined by the relations $P_0 = (1,0)$, $P_1 = \bigl( \frac{r}{q}, \frac{1}{q} \bigr)$ and $P_{i+1} = b_i P_i - P_{i-1}$. 

\section{Dimension formulas} \label{sec:dims}

Another basic invariant of a Hilbert modular surface is the dimensions of spaces of cusp forms attached to it. 

\subsection{Hilbert series}

The graded $M(\Gamma_0(\frakN))$-module $S(\Gamma_0(\frakN))$ 
%\begin{equation} S(\Gamma_0(\frakN)) \colonequals \bigoplus_{k \in 2\Z_{\geq 0}} S_{k}(\Gamma_0(\frakN)) % \end{equation}
has a Hilbert series, defined by
\begin{equation}
\label{eq:hilbert-series}
 \Hilb(S(\Gamma_0(\frakN))) \colonequals \sum_{k \in 2\Z_{\geq 0}} \dim S_{k} (\Gamma_0(\frakN))\, T^{k}\in \Z[[T]].
\end{equation}
This series gives insight into the algebraic structure of the
ring of Hilbert modular forms for~$\Gamma_0(\frakN)$.
We remind the reader that this Hilbert series sums information from all connected components (\cref{sub:hmf-def}).

There are at least two methods to compute the Hilbert series.  First, we can use the Riemann--Roch theorem \cite[Theorem 4.4]{vdG}, which needs the arithmetic invariants (arithmetic genus, cuspidal resolution, counts of elliptic points with their rotation types) computed in the next few sections.  A second independent way uses the trace
formula: see Takase \cite[Theorem 2]{Takase}, Shimizu \cite[Proposition 3.3]{Shimizu2}, 
Saito \cite[Theorem 2.1]{Saito}, or work of Okada \cite[Theorem 2.1]{Okada}.

Since it can be used as an independent check of our computation of geometric invariants, we also briefly report here on the second approach.  (More details will appear in forthcoming work of Breen--Voight \cite{BreenVoight}.)
% We recall our notation: $F$ is a totally real field with degree $n$, ring of integers~$R$, class number~$h(R)$, and narrow class number $h^+(R)$.  
We denote by~$\zeta_F$ the Dedekind zeta function of~$F$, and by~$\Nm$ the absolute norm map.  We consider pairs $(u,t)$
where $u$ ranges over $R^\times_{>0} / R^{\times 2}$, and for a fixed $u$, then $t \in R$ ranges over elements such that $t^2 - 4u$ is totally negative.
For a given pair $(u,t)$, we define the order $S(u,t) = R[x]/(x^2 - t x + u)$, and let $\frakf(u,t)$ be its conductor, which is an ideal of~$R$. 
Throughout, let $S \supseteq S(u,t)$ denote a superorder of $S(u,t)$ with class number~$h(S)$.  The trace formula (for the identity) gives the following expression:
\begin{equation} 
\label{HilbertSeriesEqn}
\begin{aligned}
\Hilb(S(\Gamma_0(\frakN)))
&= 
A \cdot T^2 
+ 
B \cdot T\Bigl(T\frac{d}{dT}\Bigr)^{n} \Bigl(\frac{T}{1-T^2}\Bigr)\\
&\quad +
\sum_{(u,t)}  C(u,t) \sum_{m\geq 1} \Nm(D_{2m - 2}(u,t)) T^{2m},
\end{aligned}
\end{equation}
where
\begin{equation}
\begin{aligned}
A & \colonequals 
(-1)^{n-1} \cdot h^+(R),
\\
B & \colonequals  \frac{1}{2^{n-1}} \cdot |\zeta_F(-1)|  \cdot h(R) \cdot \Nm(\frakN) \prod_{\frakp \mid \frakN} ( 1+ {\Nm(\frakp) }^{-1} ),
\\
C(u,t) & \colonequals \frac{1}{2} \sum_{S  \supseteq  S(u,t)}  \frac{ h(S) }{[ S^\times : R^\times] } m(\Shat,\calOhat;\calOhat^\times), 
\\
\sum_{k \geq 0} D_k(u,t) T^k & \colonequals \frac{1}{1 - t T + uT^2}.
\end{aligned}
\end{equation}
The adelic embedding numbers $m(\Shat,\calOhat;\calOhat^\times)$ were defined in \cref{sec:elliptic_pts_setup}. % In this context, we take~$\calO$ to be an Eichler order of level $\frakN$ inside the definite quaternion algebra over $F$ ramified at all infinite places. 

The term $C(u,t)$ can be further described as follows.  Let $K = F(x)/(x^2 - t x + u)$ be the CM extension of~$F$ containing $S(u,t)$. Denote its ring of integers by~$\Z_K$, its unit group by~$\Z_K^\times$, and its class number by $h(\Z_K)$. 
% For an order $S \supseteq S(u,t)$ with conductor $\frakg$ (so that $\frakg \mid \frakf(u,t)$) then 
%     \begin{equation}
%         h(S) = \frac{h(\Z_K)}{[\Z_K^\times : S^\times]} \Nm(\frakg) \prod_{\frakp \mid \frakg} \left ( 1 - \left ( \frac{K}{\frakp} \right ) {\Nm(\frakp)^{-1} } \right ).
%     \end{equation}
Then 
    \begin{align*} C(u,t) 
        & = \frac{h(\Z_K)}{2 [\Z_K^\times : R^\times]}  \sum_{\frakg \mid \frakf(u,t)}   \Nm(\frakg) \prod_{\frakp \mid \frakg} \left ( 1 - \left ( \frac{K}{\frakp} \right ) {\Nm(\frakp) }^{-1} \right )   \prod_{\frakp \mid \frakN} m_{\frakp}(S_\frakg,\calO;\calO^\times) .
    \end{align*}
where $S_\frakg$ is an order with conductor $\frakg$. 

\subsection{Explicit rational function}

We delve further to give an explicit  
expression for the Hilbert series as a rational function in $T$. For each $(u,t)$, let $\alpha(u,t)$ and $\beta(u,t)$ be the roots of the polynomial $T^2 - t T + u$. Since the discriminant is totally negative (in particular nonzero), we can write
\begin{equation}
    \frac{1}{1 - t T + uT^2} = \frac{1}{\alpha(u,t) - \beta(u,t)} \left(\frac{\alpha(u,t)}{1-\alpha(u,t)T} - \frac{\beta(u,t)}{1-\beta(u,t)T} \right).
\end{equation}
Thus, for every~$k\geq 0$,
\begin{equation}
    D_k(u,t) = \frac{1}{\alpha(u,t) - \beta(u,t)} (\alpha(u,t)^{k+1} - \beta(u,t)^{k+1}).
\end{equation}
Denote by $L$ a Galois closure for $F/\bbQ$. Note that $L$ is totally real, so the extension $L(\alpha(u,t))/L$ is of degree $2$. The characteristic polynomial of multiplication-by-$\alpha(u,t)$ on the \'etale $\bbQ$-algebra $F(\alpha(u,t))$ of degree $2n$ factors over $L$ as
\begin{equation}
    \prod_{i=1}^n (T^2 - t_i T + u_i) \in \bbQ[T]. 
\end{equation}
The roots of this polynomial can be organized into pairs $\{\alpha_i(u,t), \beta_i(u,t)\}_{i=1, \ldots, n}$, possibly with repetition. 
By examining partial fraction decompositions, we obtain
\begin{equation}
    \sum_{k \geq 0} \Nm (D_k(u,t)) T^k = \frac{1}{\Nm(\alpha(u,t) - \beta(u,t))} \sum_{\theta \in \Theta} \eps(\theta)\frac{\theta}{1 - \theta T}
\end{equation}
where 
\begin{equation}
    \Theta \colonequals \biggl\{\prod_{\gamma \in S} \gamma: S \subseteq \prod_{i=1}^n \{\alpha_i(u,t), \beta_i(u,t) \} \biggr\}
\end{equation}
and $\eps(\theta)$ is $1$ (resp.~$-1$) if $\theta$ contains an even (resp.~odd) number of $\beta_i(u,t)$. Thus,
\begin{equation} \label{normrational}
    \sum_{m\geq 1} \Nm(D_{2m-2}(u,t)) T^{2m} = T^2 \frac{1}{\Nm(\alpha(u,t) - \beta(u,t))} \sum_{\theta \in \Theta} \eps(\theta) \frac{\theta^2}{1 - \theta^2 T^2}.
\end{equation}
The other terms of the Hilbert series~\eqref{HilbertSeriesEqn} are plainly rational. We therefore have an algorithmic way of computing~$\Hilb(S(\Gamma_0(\frakN)))$ as a rational function, and hence to compute $\dim S_k(\Gamma_0(\frakN))$ for any even~$k$. 

%We will call the \defi{total degree} of a rational function $P(T)/Q(T)$ in lowest terms the quantity $\deg P(T) + \deg Q(T)$.

%\begin{prop}
%The total degree of the Hilbert series is at most
%    \begin{equation} \label{eq:HilbertDegree}
%        2 + 4(n+1) + 2^{n+1} \# \{(u,t)\}.
%    \end{equation}
%\end{prop}

%\begin{proof}
%The Hilbert series has the explicit expression
%as in \ref{HilbertSeriesEqn}. The term with coefficient~$B$ has total degree at most $4(n+1)$ 
%and the first term has total degree $2$. Each  $C(u,t)$ term is given explicitly in \cref{normrational}
%and is a rational function with denominator of degree at most $2 \cdot \#\Theta = 2^{n+1}$ and numerator of degree $2^{n+1}$. Since the expression in \cref{normrational} is the same for $(u,t)$ and $(u,-t)$, we can reduce the total degree contributed by these terms by a factor of 2.
%
%% This is proven by induction.
%Because the total degree of $f(T) + g(T)$ is at most the sum of the total degrees of $f(T), g(T)$ for any $f(T), g(T) \in \bbQ(T)$, the result is proven.
%% Actually, we can do a bit better, since D(u,t)(T) = D(u^sig,t^sig)(T). The only (u,t) in QQ are (1,0), (1, -1), and (1,1).
%\end{proof}

\section{Geometric invariants} \label{sec:geom_invs}

In this section, we compute geometric invariants for Hilbert modular surfaces~$X(\Gamma)$ as defined in \cref{sec:hilb_varieties}.

\subsection{Invariants}

Let $X$ be a smooth connected algebraic surface over $\C$.  
Regarding $X$ as a closed oriented real $4$-manifold, it has \defi{Betti numbers} $b_i \colonequals \rk H_i(X,\Z)$ satisfying $b_0 = b_4 = 1$, $b_1 = b_3$ and an \defi{Euler number} 
\begin{equation}
e \colonequals \sum_{i=0}^4 (-1)^i b_i.  
\end{equation}
As a complex K\"ahler manifold, $X$ admits \defi{Hodge numbers} $h^{p,q} \colonequals \dim_{\C} H^q(X, \Omega^p)$, where $\Omega^p$ denotes the sheaf of holomorphic $p$-forms, which satisfy $h^{p,q} = h^{q,p}$. The \defi{geometric genus} of $X$ is $p_g \colonequals h^{0,2}$ and the \defi{irregularity} is $q \colonequals h^{0,1}$. The \defi{holomorphic Euler characteristic} of $X$ is $\chi \colonequals h^{0,0} - h^{0,1} + h^{0,2}$, and the \defi{arithmetic genus} is $p_a \colonequals \chi-1$.

\subsection{Volume and Chern numbers}

We continue to follow van der Geer~\cite[Chapter IV]{vdG}. To simplify notation, write $\Gamma(1) \colonequals \rmP\!\Gamma(1)_{(1)} = \PGL_2^+(R)$. For any discrete subgroup~$\Gamma$ of~$\PGL_2^+(\R)^2$ commensurable with~$\Gamma(1)$, one can define its index as
\begin{equation}
[\Gamma(1):\Gamma] = [\Gamma(1):\Gamma(1)\cap\Gamma]/[\Gamma:\Gamma(1)\cap\Gamma]\in\Q.
\end{equation}
In particular, for any fractional ideal~$\frakb$, we have~$[\Gamma(1):\rmP\!\Gamma(1)_{\frakb}] = 1$. Then the volume of the quotient~$\Gamma\backslash\calH$ is given by the following formula:
\begin{equation} \label{eqn:volume_formula}
\begin{aligned}
\vol(\Gamma \backslash \calH) &= 2 [\Gamma(1) : \Gamma] \zeta_F(-1) = -[\Gamma(1) : \Gamma] \frac{d_F^{3/2}}{\pi^{2n}} \zeta_F(2),
\end{aligned}
\end{equation}
where~$\zeta_F$ denotes the Dedekind zeta function of~$F$.
The significance of $\vol(\Gamma \backslash \calH)$ for the computation of invariants of $X(\Gamma)$ stems from Hirzebruch's proportionality principle~\cite[IV.2.1]{vdG}. The next step is to compute the Chern numbers~$c_1^2$ and~$c_2$ of~$X(\Gamma)$. In the following statement, we say that an elliptic point of~$X(\Gamma)$ is of \defi{type} $(n;a,b)$ if its rotation factor is~$(\zeta_n^a, \zeta_n^b)$ (cf.~\cref{sec:rotationfactors}).

\begin{thm}[{\cite[Theorem IV.2.5]{vdG}}]
Let $\Gamma \leq \GL_2^+(F)$ be a congruence subgroup.
    Then the Chern numbers~$c_1^2$, $c_2$ of $X(\Gamma)$ are as follows:
    \begin{align} \label{eqn:Chern1}
      c_1^2 &= 2 \vol(\Gamma \backslash \calH) + \sum_{\textup{$\sigma$ cusp}} \sum_{k=1}^r (2 - b_{\sigma,k}) + \sum a(\Gamma; n, a, b) c(n; a,b) \, ,\\
      \label{eqn:Chern2}
      c_2 &= \vol(\Gamma \backslash \calH) + \ell + \sum a(\Gamma; n, a, b) \Bigl( \ell(n; a, b) + \frac{n-1}{n} \Bigr)
    \end{align}
    where
    \begin{itemize}
        \item
            $b_{\sigma,1}, \ldots, b_{\sigma,r}$ are the self-intersection numbers of the resolution cycle above the cusp~$\sigma$,
        \item
            $a(\Gamma; n, a, b)$ is the number of quotient singularities of $\Gamma \backslash \calH$ of type $(n; a, b)$;
        \item
            $c(n; a, b)$ is the self-intersection number of the local Chern cycle of a quotient singular of type $(n; a, b)$;
        \item
            $\ell(n; a, b)$ is the number of curves in the resolution of a quotient singularity of type $(n; a, b)$; and finally
        \item
            $\ell$ is the sum of the number of curves occurring in the resolution of cusps.
        \end{itemize}
\end{thm}

\subsection{Hodge diamond and Betti numbers} \label{sec:hodgediamond}

By Noether's formula, we have $\chi = (c_1^2 + c_2)/12$. Since there are no Hilbert modular forms of weight $(0,2)$ or $(2,0)$, we have $q = 0$, and so $p_g = \chi - 1$ and $h^{1,1} = e - 2 \chi$.

The following strong sanity checks are available for our computations. First, the holomorphic Euler characteristic~$\chi$ must be integral for every congruence subgroup. Second, holomorphic Euler characteristic values can be compared with dimensions of spaces of cusp forms, computed independently in \cref{sec:dims}. Indeed, for every~$\frakN$, we must have
\begin{equation}
\sum_{[\frakb]\in \Cl^+(R)} \chi(X_0(\frakN)_{\frakb}) = \dim S_2(\Gamma_0(\frakN)) + h^+(R).
\end{equation}

\subsection{Kodaira type} \label{sec:Kodairatype}

If $E$ is an exceptional curve on a surface $X$, i.e.~$E$ is a smooth rational curve such that $E \cdot E = -1$, then there exists a blowing-down morphism $\pi \colon X \to X'$ such that $\pi(E)$ is a point on $X'$ and $\pi$ is an isomorphism away from $E$. A surface is called \defi{minimal} if it does not contain any exceptional curve. 
We recall the Kodaira classification of minimal surfaces with irreguarity zero, in terms of their holomorphic Euler characteristic~$\chi$ and the self-intersection~$K^2$ of their canonical divisor; see \cite[VII]{vdG} for details.

\begin{table}[H]
    \centering
    \renewcommand{\arraystretch}{1.1}
    \begin{tabular}{c|l|c|c}
Kodaira dimension & \multicolumn{1}{c|}{type} & $\chi$ & $K^2$  \\ \hline \hline
         $\kappa = -1$ & rational & 1 & $8,9$ \\
\hline
         \multirow{2}{*}{$\kappa = 0$} 
                  & Enriques & 1 & 0 \\
&         K3 & 2 & 0 \\
         \hline
         $\kappa = 1$ & honestly elliptic & $\geq 1$ & 0 \\
         \hline
         $\kappa = 2$ & general type & $\geq 1$ & $\geq 1$
    \end{tabular}
    \caption{Enriques--Kodaira classification: minimal surfaces, $q = 0$}
    \label{table:Enriques-Kodaira}
\end{table}

According to \Cref{table:Enriques-Kodaira}, the invariants $\chi$ and $K^2$ are enough to determine the Kodaira dimension, except when:
\begin{enumerate}
    \item $\chi = 1$ and $K^2 = 8, 9$: rational or general type,
    \item $\chi = 1$ and $K^2 = 0$: Enriques or honestly elliptic,
    \item $\chi = 2$ and $K^2 = 0$: K3, Enriques, or honestly elliptic.
\end{enumerate}
Our computational tools give us access to holomorphic Euler characteristic $\chi$ and $K^2$, which in turn gives a list of possibilities for the Kodaira dimension. Note that in order to compute $K^2$ of a minimal model of the surface, we would need to count the number of exceptional curves on $\overline{Y}(\Gamma)$ and further blown-down surfaces. Therefore, we only use $K^2$ of the original surface as a \emph{lower} bound for~$K^2$ of the minimal model. 

One way to get more refined results towards the Kodaira classification is to use special configuration of curves on the Hilbert surface.
%Hirzebruch--Zagier divisors and study their intersections with the cusp resolutions. 
Following~\cite{vdG}, we have the following criteria in cases (1) and (2)--(3) respectively.
For (2)--(3), we rely on the notion of an elliptic configuration \cite[Definition~VII.2.8]{vdG}.

\begin{prop}[{\cite[VII.2.2]{vdG}}]\label{prop:rationality_criterion}
Let $X$ be a smooth algebraic surface with $q = 0$ satisfying either:
\begin{enumroman}
    \item $X$ contains curves $C_1$, $C_2$ such that $C_1^2 = C_2^2 = -1$ and $C_1 \cdot C_2 > 0$, or
    \item $X$ contains a curve $C$ such that $C^2 \geq 0$ and $K \cdot C < 0$.
\end{enumroman}
Then $X$ is rational.
\end{prop}

%Adding newpage to make file prettier with current text, remove anytime
% \newpage

\begin{prop}[{\cite[VII.2.9]{vdG}}] \label{prop:criterionK3Elliptic}
Let $X$ be a simply-connected non-rational algebraic surface. If $X$ contains an elliptic configuration $\mathcal C$, then $X$ is a blow-up of a K3 surface or an honestly elliptic surface. Moreover, if $X$ also contains a $(-2)$-curve $D \not\in \mathcal C$ such that $\mathcal C \cup \{D\}$ is connected, then $X$ is a blown up (elliptic) K3 surface.
\end{prop}

We have two natural sources of curves on Hilbert modular surfaces: the resolution cycles, which never contain exceptional curves, and Hirzebruch--Zagier divisors~\cite[Chapter V]{vdG}, which may sometimes contain exceptional curves.
In our implementation, we currently only use Hirzebruch--Zagier divisors for the full level $\mathfrak N = (1)$. 
%In this case, we have more refined results towards the classification using the exact values of $K^2$ and an implementation of the rationality criterion (Proposition~\ref{prop:rationality_criterion}). 
The intrinsic \texttt{RationalityCriterion} looks through the irreducible components of Hirzebruch--Zagier divisors which are exceptional, computes their intersection numbers with the cuspidal resolution cycles and checks, for any subset $S$ of disjoint exceptional curves, whether the rationality criterion is satisfied for the surface obtained by blowing down the curves in $S$. This verifies some of the results of~\cite{vdG} computationally.

\subsection{Results}

We consider Hilbert modular surfaces with $\Gamma_0^1(\mathfrak N)_{\mathfrak{b}}$-level structure for simplicity; the other level structures could be analyzed similarly.% using the tools of this paper.

\Cref{tab:chi=1} shows surfaces $X_0^1(\mathfrak N)_{\mathfrak b}$ with $\mathfrak N \neq (1)$ within our dataset in \Cref{sec:data} which could be rational according to \Cref{table:Enriques-Kodaira}. We expect this list to exhaust rational surfaces, but not all of the ones in this list to be rational. In \Cref{tab:chi=1} (and \Cref{tab:chi=2} to follow), we write $\mathfrak p_p$ for any ideal of $F$ above a prime $p\in \Z$, and the component ideal $\mathfrak b$ is described by its genus~\cite[pp.\ 3]{vdG}. In future work, we hope to check which surfaces in the table are indeed rational, by studying Hirzebruch--Zagier divisors on Hilbert modular surfaces of higher level.

\begin{table}[H]
    \centering
    \renewcommand{\arraystretch}{1.1}
    \begin{tabular}{c|c|c}
    $d_F$ & Genus of $\frakb$ & $\mathfrak N$ \\ \hline\hline
    5 & $+$ & $\p_2$, $\p_5$, $\p_3$, $\p_{11}$, $\p_2^2$, $\p_{19}$, $\p_2\p_5$, $\p_5^2$, $\p_{29}$, $\p_2\p_{11}$, $\p_{59}$ \\
    8 & $+$ & $\mathfrak p_2$, $\p_2^2$, $\p_7$, $\p_2^3$, $\p_{2}\p_7$, $\p_2^4$, $\p_{23}$ \\ 
    12 & $++$ & $\p_2$, $\p_3$, $\p_2^2$, $\p_2\p_3$, $\p_3^2$, $\p_{11}$ \\ 
    12 & $--$ & $\p_2$, $\p_3$, $\p_2\p_3$, $\mathfrak p_{11}$ \\ 
    13 & $+$ &  $\p_3$, $\p_3^2$ \\
    17 & $+$ & $\p_2$, $\p_2^2$ \\
    24 & $++$ & $\p_3$ \\ 
    28 & $++$ & $\p_2$ \\
    \end{tabular}
    \caption{Surfaces $X_0^1(\mathfrak N)_{\mathfrak{b}}$ for $\mathfrak N \neq (1)$ satisfying $\chi = 1$ and $K^2 \leq 8$}
    \label{tab:chi=1}
\end{table}

Similarly, \Cref{tab:chi=2} shows surfaces $X_0^1(\mathfrak N)_{\mathfrak b}$ with $\mathfrak N \neq (1)$ within our dataset in \Cref{sec:data} which could be K3, honestly elliptic, or general type %according to \Cref{table:Enriques-Kodaira}. 
We also expect this list to be exhaustive. Similarly to our implementation of the rationality criterion, one could implement the criterion for $X$ to be a blown-up K3 surface. This would require describing all the Hirzebruch--Zagier divisors on higher level Hilbert modular surfaces, including the genera of their irreducible components, their self-intersection numbers, and their intersection numbers with the resolution cycles. 
%We hope to finish the classification of all Hilbert modular surfaces with $\Gamma_0$ and $\Gamma_1$ level structures in future work. 

\begin{table}[h!]
    \centering
    \renewcommand{\arraystretch}{1.1}
    \begin{tabular}{c|c|c}
    $d_F$ & Genus of $\frakb$ & $\mathfrak N$ \\ \hline\hline
    5 & $+$ & $\mathfrak p_{31}$  \\
    8 & $+$ & $\p_3$, $\mathfrak p_{17}$, $\p_2^5$ \\
    12 & $++$ & $\p_2^3$, $\p_2^2 \p_3$, $\p_2^4$ \\ 
    12 & $--$ & $\p_2^2$, $\p_2^3$, $\p_2^2\p_3$ \\ 
    13 & $+$ & $\p_2$, $\p_3\overline{\p_3} = (3)$ \\
    17 & $+$ & $\mathfrak p_2 \overline{\mathfrak p_2} = (2)$, $\mathfrak p_2^3$ \\
    21 & $++$ & $\p_3$, $\mathfrak p_2$, $\mathfrak p_5$, $\p_7$, $\p_3^2$ \\
    21 & $--$ &  $\mathfrak p_3$, $\p_5$ \\
    24 & $++$ & $\p_2$, $\p_2^2$ \\
    24 & $--$ & $\p_2$ \\
    28 & $++$ & $\p_3$, $\p_2^2$, $\p_3^2$  \\
    28 & $--$ & $\p_3$ \\
    33 & $++$ &  $\mathfrak p_2$, $\p_3$, $\mathfrak p_2^2$ \\
    33 & $--$  & $\mathfrak p_2$ \\
    \end{tabular}
    \caption{Surfaces $X_0^1(\mathfrak N)_{\mathfrak{b}}$ for $\mathfrak N \neq (1)$ satisfying $\chi = 2$ and $K^2 \leq 0$}
    \label{tab:chi=2}
\end{table}

Hamahata \cite{Hamahata} has already considered the case where either $d_F=8$ or $d_F=p \equiv 1 \bmod 4$ is prime and $\Cl^+(\Z_F) = 1$; he classified some of the above surfaces using similar methods by blowing down Hirzebruch--Zagier divisors, as follows.

\begin{thm}[{ \cite[Theorems 3.5, 4.6, 5.15, 6.13]{Hamahata}}]
If $\Cl^+(\Z_F) = 1$, and $d_F$ is either a prime $p \equiv 1 \bmod 4$ or $d_F = 8$, then $X_0^1(\frakN)$ is of general type, with the following exceptions:
    \begin{itemize}
    \item For $d_F = 5$, the surfaces of levels $\frakN=\frakp_2, \frakp_5, \frakp_{11}, \frakp_2^2, \frakp_5^2 $ are rational. 
    % However, the surfaces of levels $\frakp_{29}$ and $\frakp_{59}$ are of general type. 
    Moreover, the surfaces of levels $\frakp_3, \frakp_{19}, \frakp_2 \frakp_5, \frakp_2 \frakp_{11}$ are not rational; and the surface of level $\frakp_{31}$ is not a blown-up K3 surface.
    \item For $d_F = 8$, the surfaces for levels $\frakp_2, \frakp_2^2, \frakp_2^3, \frakp_2^4$ are rational and the surface of level $\frakp_3$ is a blown-up K3 surface. % However, the surfaces of levels $\frakp_{17}, \frakp_{23}$ are of general type and 
    The surface of level $\frakp_2^5$ is a blown-up honestly elliptic surface
    % and the surface of level $\frakp_7$ is not of general type and 
    and the surface of level $\frakp_2 \frakp_7$ is not rational.
    \item For $d_F = 13$, both surfaces of levels $\frakp_3$ and $\frakp_3^2$ are rational and the surface of level $\frakp_2 = (2)$ is a blown-up K3 surface.  The surface of level $(3)$ is an honestly elliptic surface.
    \item For $d_F = 17$, both surfaces of levels $\frakp_2$ and $\frakp_2^2$ are rational, and both surfaces of levels $(2)$ and $\frakp_2^3$ are blown-up K3 surfaces.
\end{itemize}
\end{thm}

\section{Data computed} \label{sec:data}

The data resulting from our computations is publicly available online \cite{ourdata}.
We hope to include it in the LMFDB in the near future.

\subsection{Scope of data}
We computed geometric invariants $\chi$ and $K^2$ (hence the Hodge diamond) for all Hilbert surfaces attached to congruence subgroups $\Gamma_0(\frakN)_\frakb$ and~$\Gamma_0^1(\frakN)_\frakb$ for real quadratic fields $F$ with discriminant $d_F \leq 3000$ and levels~$\frakN$ with $\Nm(\frakN) \leq 5000/{d_F}^{3/2}$,
where $\frakb$ ranges over all possible narrow classes in~$\Cl^+(R)$. (This cutoff is meant as a rough approximation for the volume of the surface.) We find 4517 possibilities for $(F,\frakN,\frakb)$ in this range.
\Cref{tab:timings} displays timings for computing the geometric invariants and cusp resolutions for $\Gamma_0$ and $\Gamma_0^1$.

\begin{table}[H]
    \centering
    \renewcommand{\arraystretch}{1.1}
    \begin{tabular}{c|c|c}
        {} & $\Gamma_0$ & $\Gamma_0^1$ \\ \hline\hline
        invariants & 104.46 & 92.36 \\ \hline
        cusps & 53.99 & 52.53 \\
    \end{tabular}
    \caption{Timings for the various computations, given in CPU minutes}
    \label{tab:timings}
\end{table}

We also computed the Hilbert series for the dimensions of cuspidal spaces $S_k(\Gamma_0(\frakN))$ 
in the above range, but omitted some examples that took more than 3 hours.  We hope to return to these cases in the future. 

\subsection{Features of the data} \Cref{tab:features} presents the distribution of the different surface types in our dataset. Of the 4517 surfaces $X_0(\frakN)_\frakb$ (resp., $X^1_0(\frakN)_\frakb$), in all but 184 (resp.,~175) cases the invariants we computed allowed to completely determine the surface's Kodaira dimension. %We leave a more detailed analysis of the data for future work.

\begin{table}[H]
    \centering
    \begingroup
    \renewcommand{\arraystretch}{1.1}
    \begin{tabular}{c| c| c|c}
        Type & Kodaira dimension & $\Gamma_0$ & $\Gamma_0^1$ \\ \hline\hline
        rational & $\kappa=-1$ & 18 & 15 \\ \hline
        honestly elliptic & $\kappa=1$ & 7 & 16 \\ \hline
        general type & $\kappa=2$ & 4308 & 4311 \\ \hline
        \multirow{5}{*}{unknown} 
         & $\kappa\in \{-1,2\}$ & 61 & 44 \\
         & $\kappa\in \{0,1\}$ & 5 & 12 \\
         & $\kappa \in \{0,1,2\}$ & 51 & 43 \\
         & $\kappa\in \{1,2\}$ & 67 & 76 \\
         & \textnormal{total} & 184 & 175 \\
    \end{tabular}
    \endgroup
    \caption{Counts for the number of Hilbert modular surfaces in our dataset of each Kodaira dimension}
    \label{tab:features}
\end{table}

\subsection{Examples}
Below we present some interesting examples of surfaces that we observed in our dataset. 
\begin{example} \label{ex:disc85}
    % One component honestly elliptic ($\kappa = 1$), other component general type ($\kappa=2$):
    % 2.2.85.1-1.1-1.1-gl-0
    Let $F = \Q(\sqrt{85})$ and $\frakN = (1)$. Then $h = h^+ = 2$; a representative for the nontrivial narrow class $\frakb = (3, 1 + \sqrt{85})$.
    
    For $X_0(\frakN)_{(1)}$, we find $\chi = 4$. We initially compute $K^2 = -8$ using \eqref{eqn:Chern1}; however our model for the surface is not minimal. 
    We find that the Hirzebruch-Zagier divisors $F_1$ and $F_4$ each consists of two exceptional curves, and we can blow down successively $6+2=8$ exceptional curves to find 
    %with each component of $F_1$ intersecting a $(-2)$-curve and a $(-3)$-curve, resolving elliptic point singularities of order $2$ and $3$, respectively. 
    %After blowing these $4$ disjoint curves down, we obtain $2$ exceptional curves, each intersecting a $(-2)$-curve. 
    %Therefore, blowing down the $2$ new exceptional curves, we get $2$ exceptional curves which we blow down as well to find (after a total of $8$ blow-downs) 
    that $K^2 = 0$. From the results in \Cref{table:Enriques-Kodaira}, we conclude that $\kappa=1$ and hence the surface is honestly elliptic.

    % 2.2.85.1-1.1-3.1-gl-0
    For the other component $X_0(\frakN)_{\frakb}$, we compute that $\chi = 4$ and initially find that $K^2 = 0$. Again our model is not minimal and this time we find 2 exceptional curves as the components of the Hirzebruch-Zagier divisor $F_3$, which lead to a corrected computation of $K^2 = 2$. This implies that $\kappa=2$ and the surface is of general type. Thus we have found a Hilbert modular surface whose components belong to different Kodaira classes, and hence are not isomorphic or even birational.
\end{example}

\begin{example}
    % One where we can't distinguish between K3 or Enriques ($\kappa=0$) or honestly elliptic ($\kappa=1$).
    Let $F = \Q(\sqrt{11})$ and $\frakN = (1)$. 
    Then $h = h^+ = 2$; a representative for the nontrivial narrow class is $\frakb = (3 - \sqrt{11})$.
    % 2.2.44.1-1.1-1.1-sl-0
    Both $X^1_0(\frakN)_{(1)}$ and $X^1_0(\frakN)_{\frakb}$ are simply connected by \cite[Theroem IV.6.1]{vdG}, so we may use \Cref{prop:criterionK3Elliptic}.
    
    For~$X^1_0(\frakN)_{(1)}$, we have $\chi = 2$. We find that $F_1$ has two components, giving rise to successive blow-downs of 6 exceptional curves, as in ~\Cref{ex:disc85}, and $F_4$ is an exceptional curve intersecting a $(-2)$-curve, leading to $K^2 = -8 + 8 = 0$ via $2$ more blow-downs. We cannot determine the Kodaira dimension exactly, as \Cref{table:Enriques-Kodaira} only tells us that $\kappa \in \{0,1\}$.
    % 2.2.44.1-1.1-2.1-sl-0
    For $X^1_0(\frakN)_{\frakb}$, we have $\chi=3$ and $K^2 = -2 + 2 = 0$. Thus $\kappa = 1$ and the surface is honestly elliptic. 
\end{example}

\begin{example}
    Let $F = \Q(\sqrt{165})$ and $\frakN = (1)$. In this case, $\Cl(R) \iso C_2$ and $\Cl^+(R) \iso C_2 \times C_2$, where $C_2$ is the cyclic group of order 2.
    % 2.2.165.1-1.1-1.1-sl-0
    % 2.2.165.1-1.1-1.1-gl-0
    Computing for~$X_0^1(\frakN)_{(1)}$, we find that $\chi = 4$. The existence of 20 exceptional curves coming from Hirzebruch-Zagier divisors gives $K^2 = -20 + 20 = 0$. Thus $\kappa=1$ and the surface is honestly elliptic. By contrast, for $X_0(\frakN)_{(1)}$ we have $\chi = 3$ and $K^2 = -10 + 20 = 10$, and hence $\kappa=2$ and the surface is of general type. (We note in passing that, for both $X_0$ and $X_0^1$, the remaining 3 components are all of general type.)
\end{example}

\begin{example} \label{ex:rational}
    When $F = \Q(\sqrt{13})$ and $\mathfrak N = (1)$, we verify that the Hilbert modular surface is rational. We consider the Hirzebruch--Zagier divisor $F_3$, which in this case consists of a single curve, and compute that we have the following configuration of curves in both cases, where the thick lines are curves in the resolutions of the cusps.

    
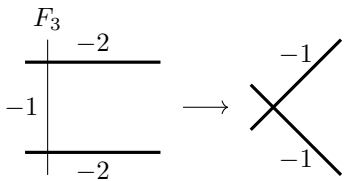
\begin{center}
        \begin{tikzpicture}[scale=0.6]
        \draw[very thick] (0,0) -- (3,0);
        \draw[very thick] (0,2) -- (3,2);
        \draw (0.5,-0.5) -- (0.5,2.5) node[above] {$F_3$};
        \draw (0.5, 1) node[left] {$-1$};
        \draw (1.5, 2) node[above] {$-2$};
        \draw (1.5, 0) node[below] {$-2$};

        \draw[->] (3.5, 1) -- (4.5, 1);

        \draw[very thick] (5, 0.5) -- (7, 2.5);
        \draw[very thick] (5, 1.5) -- (7, -0.5);
        \draw (6, 1.75) node[above] {$-1$};
        \draw (6, 0.25) node[below] {$-1$};
        \end{tikzpicture}
    \end{center}
    We blow down the exceptional curve $F_3$ and apply \Cref{prop:rationality_criterion} to conclude that the surface is rational. This process is excecuted in \texttt{RationalityCriterion}.
\end{example}

\subsection{Future directions}

An immediate goal is to gather data for other standard congruence subgroups, namely $\Gamma_1(\frakN)_\frakb$, $\Gamma_1^1(\frakN)_\frakb$ and $\Gamma(\frakN)_{\frakb}$.
The case of~$\Gamma_1(\frakN)_{\frakb}$ for non-squarefree levels seems more subtle, as the stabilizer of a cusp does not seem to be of the exact form~$G(M,V)$, and only contains such a group with finite index.
It remains to determine how this finite quotient acts on the resolution cycles.

We also hope to complete the classification into Kodaira types of the surfaces we have computed and resolve any ambiguities by studying Hirzebruch-Zagier divisors on higher level modular surfaces. Van der Geer proves \cite[Theorem VII.3.3]{vdG} that the Hilbert modular surface 
$X_0^1(1)_{\frakb}$
over a real quadratic field $F$ is of general type in all but finitely many cases. A key result in proving this theorem are estimates \cite[Theorem VII.5.1]{vdG} which show that if $d_F > 500$, then $X_0^1(1)_{\frakb}$ is of general type. It seems that an analogous result should be true in our more general setting of higher level and further variants ($\Gamma_0, \Gamma_0^1, \Gamma_1, \Gamma_1^1$, etc.)~of Hilbert modular groups. Proving such a result would guarantee that \Cref{tab:chi=1} and \Cref{tab:chi=2} are exhaustive.

%. We aim to determine each surface's Kodaira dimension exactly by extending our results on Hilbert series from \cref{sec:dims}. This involves studying algorithmically which cusp forms of higher weights extend to the resolution cycles above elliptic points. We also plan to further study Hirzebruch-Zagier divisors on higher level modular surfaces in order to extend our rationality criterion and to distinguish between honestly elliptic, K3, and Enriques surfaces.

\bibliographystyle{amsalpha}

\begin{thebibliography}{999}

\bibitem{ourcode}
Eran Assaf, Angelica Babei, Ben Breen, Sara Chari, Edgar Costa, Juanita Duque-Rosero, Aleksander Horawa, Jean Kieffer, Avinash Kulkarni, Grant Molnar, Abhijit Mudigonda, Michael Musty, Sam Schiavone, Shikhin Sethi, Samuel Tripp, and 
    John Voight, Hilbert modular forms, 2023, \url{https://github.com/edgarcosta/hilbertmodularforms/}.

\bibitem{ourdata}
Eran Assaf, Angelica Babei, Ben Breen, Edgar Costa,
Juanita Duque-Rosero, Aleksander Horawa, Jean Kieffer,
Avinash Kulkarni, Grant Molnar, Sam Schiavone, and John Voight, Hilbert modular surfaces data, 2023, \url{https://github.com/edgarcosta/hilbertmodularsurfacesdata/}.

\bibitem{Magma}
Wieb Bosma, John Cannon, and Catherine Playoust, \emph{The Magma algebra system. I. The user language}, J.\ Symbolic Comput. \textbf{24} (1997), no.\ 3--4, 235--265.

\bibitem{BreenVoight}
Ben Breen and John Voight, \emph{Computing Hilbert modular forms via the trace formula}, 2023, preprint.

\bibitem{DasguptaKakde}
Samit Dasgupta and Mahesh Kakde, \emph{On constant terms of Eisenstein series}, Acta Arith.\ \textbf{200} (2021), no.~2, 119--147.

\bibitem{DembVoight}
Lassina Demb\'el\'e and John Voight, \emph{Explicit methods for Hilbert modular forms}, Elliptic curves, Hilbert modular forms and Galois deformations, Birkhauser, Basel, 2013, 135--198.

\bibitem{Freitag}
Eberhard Freitag, \emph{Hilbert modular forms}, Springer-Verlag, Berlin, 1990.

\bibitem{Goren}
Eyal Z.\ Goren, \emph{Lectures on Hilbert modular varieties and modular forms}, CRM Monograph Series, vol.~14, Amer.\ Math.\ Soc., Providence, RI, 2002.

\bibitem{Grundman}
Helen Grundman, \emph{Hilbert modular variety computations}, WIN -- women in numbers, Fields Inst.\ Commun., vol.~60, Amer. Math. Soc., Providence, RI, 2011, 3--14.

\bibitem{Hamahata}
Yoshinori Hamahata, \emph{Hilbert Modular Surfaces with $p_g \le 1$}, Mathematische Nachrichten (173) \textbf{1} (1995), 193--236.

\bibitem{Hirzebruch}
Friedrich E.\ P.\ Hirzebruch, \emph{Hilbert modular surfaces}, Enseign.\ Math.\ (2) \textbf{19} (1973), 183--281.

\bibitem{Hir77}
F.~Hirzebruch, \textit{The ring of Hilbert modular forms for real quadratic number fields of small discriminant}, Lecture Notes in Math., vol.~627, Springer-Verlag, Berlin, 1977.

\bibitem{vdg-Hirzebruch}
Friedrich Hirzebruch and Gerard van der Geer, \emph{Lectures on Hilbert modular surfaces}, S\'{e}minaire de Math\'{e}matiques Sup\'{e}rieures, vol.~77, Presses de l'Universit\'{e} de Montr\'{e}al, Montr\'eal, 1981.

\bibitem{Johansson}
Stefan Johansson, \emph{Genera of arithmetic Fuchsian groups}, Acta Arith.\ \textbf{86} (1998), no.~2, 171--191. 

\bibitem{LMFDB}
The LMFDB Collaboration, \emph{The {L}-functions and Modular Forms Database},  \\ \texttt{http://www.lmfdb.org}, 2023.

\bibitem{Okada}
Kaoru Okada, \emph{Hecke eigenvalues for real quadratic fields}, Experiment. Math.\
\textbf{11} (2002), 407--426.

\bibitem{Prestel}
Alexander Prestel, \emph{Die elliptischen Fixpunkte der Hilbertschen Modulgruppen}, Math.\ Ann.\ \textbf{177} (1968), 181--209. 

\bibitem{Saito}
Hiroshi Saito, \emph{On an operator $U_{\chi }$ acting on the space of Hilbert cusp forms}, J. Math. Kyoto Univ.\
\textbf{24} (1984), 285--303.

\bibitem{Shimizu2}
Hideo Shimizu, \emph{On zeta functions of quaternion algebras} Ann.\ of Math.~(2) \textbf{81} (1965), 166--193.

\bibitem{Stanley}
Richard P.\ Stanley, \emph{Enumerative combinatorics: Volume 1}, 2nd.\ ed., Cambridge Studies in Advanced Mathematics, vol.~49, Cambridge University Press, Cambridge, 2012.

\bibitem{Takase}
Koichi Takase, \emph{On the trace formula of the Hecke operators and the special values of the second L-functions attached to the Hilbert modular forms}, Manuscripta Math.\ \textbf{55} (1986), no.~2, 137--170.

\bibitem{vdG}
Gerard van der Geer, \textit{Hilbert modular surfaces}, Ergeb.\ Math.\ Grenzgeb.\ (3), vol.~16, Springer-Verlag, Berlin, 1988. 

\bibitem{Voight}
John Voight, \emph{Quaternion algebras}, Grad.\ Texts in Math., vol.~288, Springer, Cham, 2021.

\end{thebibliography}

\end{document}